\documentclass[11pt,reqno]{amsart}

\usepackage[margin=1in]{geometry}

\usepackage{amssymb, mathrsfs, amscd, bbm}
\usepackage{extpfeil}

\usepackage{graphicx}
\usepackage{colortbl}
\usepackage{float}
\usepackage{caption}
\usepackage{subcaption} 
\graphicspath{{Figure/}}
\usepackage{comment}
\usepackage{enumerate}
\usepackage{url}
\usepackage{hyperref}
\usepackage{doi}
\usepackage{kotex} 

\newtheorem{theorem}{Theorem}[section]
\newtheorem{lemma}{Lemma}[section]

\newtheorem{proposition}{Proposition}[section]
\newtheorem{remark}{Remark}[section]

\numberwithin{equation}{section} 

\newcommand{\R}{\mathbb{R}}
\newcommand{\e}{\varepsilon}
\newcommand{\pa}{\partial}

\newcommand{\norm}[1]{\left\lVert#1\right\rVert}

\makeatletter
\@namedef{subjclassname@2020}{\textup{2020} Mathematics Subject Classification}
\makeatother

\title[Outflow problem for the NSK equations]{Time-asymptotic stability of viscous shocks for the outflow problem of one-dimensional compressible fluids of Korteweg type}

\thanks{}

\author[S. Han]{Sungho Han}
\address[Sungho Han]{Department of Mathematical Sciences, \newline
Korea Advanced Institute of Science and Technology (KAIST), \newline
Daejeon 34141, Korea}
\email{sungho\_han@kaist.ac.kr}

\author[J. Kim]{Jeongho Kim}
\address[Jeongho Kim]{Department of Applied Mathematics, \newline 
Kyung Hee University, 1732, Deogyeong-daero, Giheung-gu, \newline
Yongin-si, Gyeonggi-do 17104, Republic of Korea}
\email{jeonghokim@khu.ac.kr}

\author[H. Oh]{HyeonSeop Oh}
\address[HyeonSeop Oh]{Department of Mathematical Sciences, \newline
Korea Advanced Institute of Science and Technology (KAIST), \newline
Daejeon 34141, Korea}
\email{ohs2509@kaist.ac.kr}

\begin{document}

\bibliographystyle{acm}

\thanks{\textbf{Acknowledgment.} S. Han and H. Oh were partially supported by the National Research Foundation of Korea (RS-2024-00361663 and NRF-2019R1A5A1028324), and S. Han was additionally supported by the Doomyung Fellowship. J. Kim was supported by Samsung Science and Technology Foundation under Project Number SSTF-BA2401-01. The authors thank Professor Moon-Jin Kang for his valuable comments.}

\begin{abstract} 
We study the time-asymptotic stability of viscous-dispersive shock waves for the outflow problem of the barotropic Navier--Stokes--Korteweg equations, which describe viscous fluids with internal capillarity. Assuming that the far-field state is subsonic or transonic and that the velocity at the boundary is larger than the far-field velocity, we prove that the solution converges to the corresponding viscous-dispersive shock wave as $t \to +\infty$, provided that the shock amplitude and the initial perturbation are sufficiently small. The proof is based on the method of $a$-contraction with shifts (for viscous equations) introduced in \cite{KV17,KV21,KVW23}. A main difficulty comes from controlling the boundary effect of the viscous-dispersive shock wave, as well as the influence of capillarity near the boundary.
\end{abstract}

\maketitle

\section{Introduction}
\setcounter{equation}{0}

The compressible Navier--Stokes--Korteweg (NSK) equations are a fundamental model for describing the motion of compressible fluids exhibiting liquid-vapor phase transitions \cite{DS85, Kor01, VdW94}. To mathematically account for the capillarity effects and surface tension at phase interfaces, the NSK equations extend the standard compressible Navier--Stokes (NS) equations by incorporating a higher-order dispersive term, known as the Korteweg stress tensor \cite{BDD06, Kotschote08, Kotschote10}.

In this paper, we consider the initial-boundary value problem (IBVP) for the one-dimensional barotropic NSK equations in Eulerian coordinates on the half-line $\mathbb{R}_+ := (0, \infty)$:
\begin{align} \label{eq:NSK}
\begin{aligned}
    &\rho_t +(\rho u)_x=0, \quad t>0, \quad x>0, \\
    &(\rho u)_t +(\rho u^2+p(\rho))_x = \mu u_{xx} + \kappa \rho\rho_{xxx},
\end{aligned}
\end{align}
where $\rho=\rho(t,x)>0$ and $u=u(t,x)$ denote the density and velocity of the fluid, and $p(\rho)$ is the pressure. The constants $\mu > 0$ and $\kappa > 0$ represent the viscosity and capillarity coefficients, respectively. We assume the standard $\gamma$-law pressure $p(\rho)=\rho^\gamma$ with $\gamma > 1$. The third-order derivative term can be rewritten by using the Korteweg stress tensor $K(\rho)$ as $\rho\rho_{xxx} = \left(\rho\rho_{xx}-\frac{1}{2}(\rho_x)^2\right)_x =: K(\rho)_x$. 

We impose the initial condition
\begin{equation*}
    (\rho, u)(0,x) = (\rho_0(x), u_0(x)), \quad x \in \mathbb{R}_+,
\end{equation*}
with the non-vacuum condition $\inf_{x\in\mathbb{R}_+}\rho_0(x)>0$, and the constant far-field condition
\begin{equation*}
    \lim_{x\to\infty} (\rho(t,x), u(t,x)) = (\rho_+, u_+), \quad \rho_+>0.
\end{equation*}
At the boundary $x=0$, appropriate boundary values should be prescribed, depending on the type of the initial boundary value problem. We specifically focus on the outflow problem, where the fluid flows out from the domain. This corresponds to a strictly negative velocity at the boundary $x=0$:
\begin{equation}\label{eq:BC-outflow}
    u(t,0) = u_- < 0.
\end{equation}
Due to the presence of the higher-order capillarity term in the NSK equations, an additional boundary condition is required to ensure the well-posedness of the problem. Accordingly, we impose the boundary condition on the derivative of the density to close the system:
\begin{equation}\label{eq:BC-capillarity}
    \rho_x(t,0) = 0.
\end{equation}
Physically, the condition $\rho_x(t,0) = 0$ corresponds to a neutral wetting condition in the context of diffuse-interface models. As established in the Cahn--Hilliard theory, this condition implies a vanishing variation of the fluid-solid surface energy, meaning the fluid interface intersects the wall at a $90^\circ$ contact angle with no preferential affinity for either phase \cite{Seppecher93, Seppecher96}. Through the variational approach, this thermodynamic principle can be extended to the Navier--Stokes--Korteweg framework, where it translates into the homogeneous Neumann boundary condition to prevent any boundary-driven capillary mass flux \cite{Kotschote08, Kotschote10}.

The main goal of this paper is to establish the time-asymptotic stability of right-going (i.e., moving from left to right) viscous-dispersive (viscous-capillary) shock profiles for the outflow problem of the NSK equations. It is well known that the time-asymptotic stability of wave patterns for the barotropic Navier--Stokes (NS) and the NSK equations is closely related to the Riemann problem of the Euler system.

For the whole-space problem, where the spatial domain is $\R$, the time-asymptotic stability of viscous shocks, rarefactions, and their composite waves for the NS equations has been extensively studied; we refer to \cite{HKK23, KVW23, Liu86, MN85, MN86, SzepessyXin93} and the references therein. Corresponding results for the NSK equations have been obtained in \cite{C12,CHZ15,HKKL25-JDE,HK26}, despite the analytical difficulties presented by the nonlinear third-order dispersion term.

Turning to the initial-boundary value problem, a boundary layer solution may appear as an asymptotic profile due to the presence of the boundary. As classified by Matsumura \cite{MatBVP}, the asymptotic profile of the NS equation is determined by the boundary and far-field states, leading to various wave patterns: the boundary layer solution, the viscous shock, the rarefaction wave, or their superpositions. Based on this classification, the time-asymptotic stability of boundary layer solutions, rarefaction waves, and their superpositions for the outflow problem of the NS equations was studied in \cite{KNZ03, KZ08, KZ09}. Similar developments for the NSK equations were also obtained in \cite{Hong20, LTY22, LXC23, LZ21}.

However, the stability of viscous-dispersive shock waves for the outflow problem of the NSK equations has remained open due to technical difficulties caused by the outflow boundary condition. The classical anti-derivative method, combined with the shift argument, has been successfully applied to the impermeable wall and inflow problems for the NS equations \cite{HMS03, MM99, MN01} and the NSK equations \cite{CLS19, Hong22}. However, it cannot be directly applied to the outflow problem due to the absence of a density boundary condition, which makes it difficult to define the anti-derivative variables and the constant shift as in \cite{MM99}. To overcome this limitation, we employ the so-called the method of $a$-contraction with shifts, developed by \cite{KV17, KV21, KV-Inven, KVW23} for viscous equations. This $L^2$-based technique effectively controls shock perturbations and boundary terms without relying on anti-derivatives. Furthermore, this approach has been extended to initial-boundary value problems for the NS equations \cite{HKKKO26,HKKL25-JMAA, KOW25}. Motivated by these developments, we apply the method of $a$-contraction with shifts to the outflow problem of the NSK equations and establish the time-asymptotic stability of viscous-dispersive shock waves.

Regarding the asymptotic profile, we focus on right-going weak viscous-dispersive shocks. Specifically, we define the subsonic and transonic regime as
$$\Omega_{sub}^- := \{(\rho, u) \mid  \lambda_2(\rho,u) > 0, \, u < 0\}, \quad  \text{and} \quad \Gamma_{trans}^- := \{(\rho, u) \mid  \lambda_2(\rho,u) = 0\},$$
where $\lambda_2(\rho,u) := u + \sqrt{p'(\rho)}$ denotes the second eigenvalue of the hyperbolic part of the NSK equations. We assume that the right-end state $(\rho_+, u_+)$ with $u_+ < u_-$ belongs to either the subsonic or transonic regime:
\begin{equation}\label{cond:U+}
(\rho_+, u_+) \in \Omega_{sub}^- \cup  \Gamma_{trans}^-.
\end{equation}

For a fixed $(\rho_+, u_+) \in \Omega_{sub}^- \cup \Gamma_{trans}^-$ and a given boundary velocity $u_- \in (u_+, 0)$, the left state $\rho_-$ and the shock speed $\sigma$ are uniquely prescribed by the Rankine-Hugoniot (RH) conditions:
\begin{equation}
\begin{aligned}\label{RH}
&-\sigma (\rho_+ - \rho_-) + (\rho_+ u_+ - \rho_- u_-) = 0,\\
&-\sigma (\rho_+ u_+ - \rho_- u_-) + (\rho_+ u_+^2 - \rho_- u_-^2 + p(\rho_+) - p(\rho_-)) = 0,
\end{aligned} 
\end{equation}
along with the Lax entropy condition 
\begin{equation} \label{lax}
    \lambda_2(\rho_+,u_+) < \sigma < \lambda_2(\rho_-,u_-).
\end{equation} Consequently, the 2-shock speed $\sigma$ is strictly positive and given by:
\[
\sigma = u_+ + \sqrt{\frac{\rho_-}{\rho_+}}\sqrt{\frac{p(\rho_+) - p(\rho_-)}{\rho_+ - \rho_-}} > 0.
\]

The corresponding viscous-dispersive 2-shock profile $(\tilde{\rho}, \tilde{u})(\xi)$, defined on $\mathbb{R}$, with the traveling wave variable $\xi = x - \sigma t$ is governed by the following ODE system:
\begin{equation}
\begin{aligned}\label{eq:VCshock}
&\begin{cases}
-\sigma \tilde{\rho}' + (\tilde{\rho} \tilde{u})' = 0,  \qquad  ':= \frac{d}{d\xi},\\
-\sigma (\tilde{\rho} \tilde{u})' +(\tilde{\rho} \tilde{u}^2 + p(\tilde{\rho}))' = \mu \tilde{u}'' + \kappa K(\tilde{\rho})',\\
(\tilde{\rho}, \tilde{u})(\pm\infty)=(\rho_\pm, u_\pm). 
\end{cases}\\
\end{aligned} 
\end{equation}

\subsection{Main Results} 
We now state our main theorem regarding the global well-posedness and the time-asymptotic stability of the outflow problem \eqref{eq:NSK} toward the viscous-dispersive shock.

\begin{theorem}\label{thm:main}
For a given right-end state $(\rho_+, u_+) \in \mathbb{R}_+ \times \mathbb{R}$ satisfying $u_+ < 0$ and \eqref{cond:U+}, there exist positive constants $\delta_0$ and $\varepsilon_0$ such that the following holds:

For any boundary velocity $u_- < 0$ satisfying $u_+ < u_-$ and $|u_- - u_+| < \delta_0$, let $\rho_- > 0$ be the unique constant state such that $(\rho_-, u_-)$ satisfying \eqref{RH} and \eqref{lax} with $(\rho_+,u_+)$. Let $(\tilde{\rho}, \tilde{u})(x - \sigma t)$ be the viscous-dispersive 2-shock wave connecting $(\rho_-, u_-)$ to $(\rho_+, u_+)$ with $\tilde{\rho}(0)= (\rho_- + \rho_+)/2$.

Then, there exists $\beta > 0$ large enough, depending only on the shock strength $|u_- - u_+|$, such that the following holds: Let $(\rho_0, u_0)$ be any initial data satisfying
\begin{equation*}
\begin{aligned}
    \|(\rho_0, u_0) - (\rho_+, u_+)\|_{L^2(\beta, \infty)} + \|(\rho_0, u_0) - (\rho_-, u_-)\|_{L^2(0, \beta)} + \|(\partial_x \rho_0, \partial_x u_0)\|_{H^1(\mathbb{R}_+) \times L^2(\mathbb{R}_+)} < \varepsilon_0.
\end{aligned}
\end{equation*}
Then, there exists a Lipschitz continuous function $X:[0,\infty)\to\R$, called the shift, such that the outflow problem \eqref{eq:NSK} admits a unique global-in-time solution $(\rho, u)(t,x)$ satisfying 
\begin{equation*}
    \begin{aligned}
    &\rho(t,x) - \tilde{\rho}(x - \sigma t - X(t) - \beta) \in C([0,\infty);H^2(\mathbb{R}_+)), \quad (\rho - \tilde{\rho})_x \in L^2(0,\infty;H^2(\mathbb{R}_+)),\\
    &u(t,x) - \tilde{u}(x - \sigma t - X(t) - \beta) \in C([0,\infty);H^1(\mathbb{R}_+)), \quad (u - \tilde{u})_x \in L^2(0,\infty;H^1(\mathbb{R}_+)).
\end{aligned}
\end{equation*}

Moreover, the solution asymptotically converges to the viscous-dispersive shock wave, and the speed of shift tends to zero:
\begin{equation}\label{asym-U}
	\begin{aligned}
	&\lim_{t\to\infty} \|\rho(t,\cdot)-\tilde{\rho}(\cdot-\sigma t-X(t)-\beta)\|_{W^{1,\infty}(\R_+)}=0,\\	
	&\lim_{t\to\infty} \|u(t,\cdot)-\tilde{u}(\cdot-\sigma t-X(t)-\beta)\|_{L^{\infty}(\R_+)}=0,
	\end{aligned}
\end{equation}
\begin{equation*}
    \lim_{t \to \infty} |\dot{X}(t)| = 0.
\end{equation*}
\end{theorem}

\subsection{Main ideas of the proof}
The proof of Theorem \ref{thm:main} is based on the method of $a$-contraction with shifts. Although this method is technically more involved in Lagrangian mass coordinates than in Eulerian coordinates, we adopt the Eulerian framework since the former leads to a free boundary problem, see \cite[Section 4]{KNZ03}. To apply this strategy to the NSK equations, we proceed with the following two main steps.

First, to handle the higher-order Korteweg term effectively and to ensure the compatibility of the boundary conditions, we reformulate the NSK equations \eqref{eq:NSK} into an augmented system of conservation laws \eqref{eq:augmented-sys} by introducing an auxiliary capillarity variable $w := \sqrt{\kappa/\rho}\rho_x$. A key step in our analysis is utilizing this augmented formulation to perform the weighted relative entropy estimates directly in the Eulerian coordinates. 
We then construct a weight function $a(t,x)$ and a dynamical shift $X(t)$ to control the bad terms localized by a shock profile.

Second, to close the a priori estimates, we need to recover the dissipation terms for the capillarity variable $w$. This is achieved by deriving higher-order energy estimates together with delicate cross-estimates (e.g., between the velocity and capillarity perturbations), which provide the necessary dissipations without requiring additional boundary conditions (See Lemma~\ref{lem:higher2}).

The rest of this paper is organized as follows. In Section~\ref{sec:prelim}, we recall the properties of viscous-dispersive shock profiles and reformulate the NSK equations into the augmented system. In Section~\ref{sec:apriori}, we state the local existence of solutions and construct the weight function and the dynamical shift to be used throughout the paper. We then present the a priori estimates, which lead to the proof of Theorem~\ref{thm:main} via a continuation argument. Section~\ref{sec:L2_energy} provides the proof of $L^2$-estimates measured by a weighted relative entropy for the augmented system. 
In Section~\ref{sec:higher_order}, we derive the higher-order energy estimates for the perturbations. In particular, we perform cross-estimates to recover dissipations for the capillarity variable, thereby closing the global energy estimates. Finally, in Appendix~\ref{sec:app-A}, we provide the technical details for the time-asymptotic convergence toward the shifted viscous-dispersive shock wave.

\section{Preliminaries}\label{sec:prelim}
\setcounter{equation}{0}
In this section, we gather some preliminary contents on the viscous-dispersive shock wave and the extended system for the NSK equations.

\subsection{Viscous-dispersive shock profiles}
First of all, we recall the existence and properties of the viscous-dispersive shock wave for the NSK equations, where the shock strength is small.

\begin{lemma}\label{lem:shock-profile}\cite{HKKL25-JDE}
For any given state $(\rho_+, u_+)$, there exist positive constants $\delta_0, c,$ and $C$ such that the following holds. For any left-end state $(\rho_-, u_-)$ satisfying \eqref{RH}, \eqref{lax} and $\delta := |u_- - u_+|\sim |\rho_- - \rho_+| \leq \delta_0$, there exists a unique solution $(\tilde{\rho}, \tilde{u})(\xi) = (\tilde{\rho}, \tilde{u})(x-\sigma t)$ to \eqref{eq:VCshock} with $\tilde{\rho}(0)=\frac{\rho_- + \rho_+}{2}$.

Furthermore, the profile satisfies the monotonicity properties
\[
    \tilde{u}_\xi < 0, \quad \tilde{\rho}_\xi < 0,
\]
and the following decay estimates:
\begin{align*}
    &|\tilde{u}(\xi) - u_\pm| \leq C\delta e^{-c\delta |\xi|}, \quad \text{for } \pm \xi > 0, \\
    &|(\tilde{\rho}_\xi(\xi), \tilde{u}_\xi(\xi))| \leq C\delta^2 e^{-c\delta |\xi|}, \quad \forall \xi \in \mathbb{R}, \\
    &|(\tilde{\rho}_{\xi \xi}(\xi), \tilde{u}_{\xi \xi}(\xi))| \leq C\delta |(\tilde{\rho}_\xi(\xi), \tilde{u}_\xi(\xi))|, \quad \forall \xi \in \mathbb{R}.
\end{align*}
\end{lemma}

\begin{remark}
	Although the precise statement in \cite{HKKL25-JDE} is given in terms of the specific volume $v$ and the velocity $u$ in the Lagrangian mass coordinate, it is straightforward to adapt the existence and properties of viscous-dispersive shock wave in terms of $(\rho,u)$.
\end{remark}

\begin{remark}
    By Lemma \ref{lem:shock-profile} and the Lax entropy condition \eqref{lax}, the velocity component of the viscous-dispersive 2-shock with a small amplitude satisfies
    \begin{equation*}
    |\sigma -\tilde{u}(\xi) - \sqrt{p'(\tilde{\rho}(\xi))} | \le C \delta,
    \end{equation*} 
    and consequently $\sigma-\tilde{u}(\xi)>0$ for all $\xi\in \R$.
\end{remark}

\subsection{Augmented system and boundary conditions}
As we mentioned in the introduction, the main tool for obtaining stability is based on the relative entropy estimate. To exploit the relative entropy estimate and to handle the capillarity tensor, we rewrite the NSK equations \eqref{eq:NSK} by introducing a new variable:
\begin{equation}\label{def:w}
    w := \sqrt{\frac{\kappa}{\rho}} \rho_x,
\end{equation}
which is motivated by \cite{BDD06,HKKL25-JDE}. Using this variable, the capillarity tensor can be rewritten as
\[
    \kappa K(\rho) = \kappa\left(\rho\rho_{xx} - \frac{1}{2}(\rho_x)^2\right) = \sqrt{\kappa} \rho^{3/2}w_x.
\]
Differentiating the continuity equation \eqref{eq:NSK}$_1$ with respect to $x$ and utilizing \eqref{def:w}, we find that $w$ satisfies:
\[
    (\rho w)_t + (\rho u w)_x = -\sqrt{\kappa} (\rho^{3/2}u_x)_x.
\]
Consequently, the original NSK equations \eqref{eq:NSK} can be reformulated as the following augmented system in the standard form of viscous conservation law:
\begin{align}\label{eq:augmented-sys}
\begin{aligned}
    &\rho_t + (\rho u)_x = 0, \\
    &(\rho u)_t + (\rho u^2 + p(\rho))_x =\mu u_{xx} + \sqrt{\kappa} (\rho^{3/2} w_x)_x, \\
    &(\rho w)_t + (\rho u w)_x = -\sqrt{\kappa} (\rho^{3/2} u_x)_x.
\end{aligned}
\end{align}

By expressing the augmented system in this standard form, we can exploit its underlying structural properties. In particular, this formulation enables us to utilize the standard theory of relative entropy method in Lemma \ref{lem:weighted-rel-entropy} and provides a cancellation between the second-order terms for $u$ and $w$, which plays a key role in closing the $H^1$-energy estimates (see \eqref{eq:higher-order-1}). Furthermore, the boundary condition \eqref{eq:BC-capillarity} can be expressed in terms of $w$ as a simple Dirichlet boundary condition $w(t,0)=0$. Similarly, we define the auxiliary variable for the shock profile as $\tilde{w}:=\sqrt{\frac{\kappa}{\tilde{\rho}}}\tilde{\rho}_\xi$.

\section{A Priori Estimates and Proof of Theorem \ref{thm:main}} \label{sec:apriori}
\setcounter{equation}{0}

To prove the global existence and time-asymptotic stability of the viscous-dispersive shock in Theorem \ref{thm:main}, we combine the local well-posedness with an a priori energy estimates.

\subsection{Local existence of solutions} 
First, we state the local-in-time existence of the outflow problem for the NSK equations. 
The proof of local existence is based on standard iteration argument and is thus omitted (see, for example, \cite{HKKL25-JMAA}).

\begin{proposition}[Local existence]\label{prop:loc} 
    For any constant $\beta > 0 $, let $\underline{\rho}$ and $\underline{u}$ be smooth monotone functions such that
    \[
    (\underline{\rho}(x), \underline{u}(x)) = (\rho_+, u_+), \quad \text{for } x \geq \beta, \quad \underline{\rho}(0) > 0.
    \]
    Given any positive constants $M_0, M_1, \underline{\kappa}_0, \overline{\kappa}_0,  \underline{\kappa}_1, \overline{\kappa}_1$ satisfying $M_0 < M_1$ and $\underline{\kappa}_1 < \underline{\kappa}_0 < \overline{\kappa}_0 < \overline{\kappa}_1$, there exists a time $T_0 > 0$ such that if the initial data satisfies
    \begin{align*}
        &\norm{\rho_0 - \underline{\rho}}_{H^2(\mathbb{R}_+)} + \norm{u_0 - \underline{u}}_{H^1(\mathbb{R}_+)} \leq M_0,\\
        &0 < \underline{\kappa}_0 \leq \rho_0(x) \leq \overline{\kappa}_0, \quad \text{for all } x \in \mathbb{R}_+,
    \end{align*}
    then the outflow problem \eqref{eq:NSK}-\eqref{eq:BC-outflow}-\eqref{eq:BC-capillarity} has a unique strong solution $(\rho, u)$ on $[0,T_0]$ satisfying:
\begin{equation*}
\begin{aligned}
    &\rho - \underline{\rho} \in C([0,T_0];H^2(\mathbb{R}_+)) \cap L^2(0,T_0;H^3(\mathbb{R}_+)),\\
    &u - \underline{u} \in C([0,T_0];H^1(\mathbb{R}_+)) \cap L^2(0,T_0;H^2(\mathbb{R}_+)),
\end{aligned} 
\end{equation*}
and
\[
    \norm{(\rho - \underline{\rho})(t,\cdot)}_{H^2(\mathbb{R}_+)} + \norm{(u - \underline{u})(t,\cdot)}_{H^1(\mathbb{R}_+)} \leq M_1, \quad \forall t \in [0, T_0].
\]
Moreover, the solution satisfies
\[
    0<\underline{\kappa}_1 \leq \rho(t,x) \leq \overline{\kappa}_1, \quad u(t,0) = u_-, \quad \rho_x(t,0) = 0, \quad \forall t \in (0, T_0], \quad x \in \mathbb{R}_+.
\]
\end{proposition}

\subsection{Construction of the weight and dynamical shift}
For simplicity, we use the following notations if there is no confusion:
\begin{align*}
    \tilde{\rho}(t,x) &:= \tilde{\rho}^{X,\beta}(t,x) = \tilde{\rho}(x-\sigma t - X(t) - \beta),\\
    \tilde{u}(t,x) &:= \tilde{u}^{X,\beta}(t,x) = \tilde{u}(x-\sigma t - X(t) - \beta),\\
    \tilde{w}(t,x) &:= \tilde{w}^{X,\beta}(t,x) = \tilde{w}(x-\sigma t - X(t) - \beta),
\end{align*}
where the shift function $X(t)$ and the constant $\beta$ will be chosen later. Then, we introduce a weight function $a(t,x) = a^{X,\beta}(t,x)$ as
\begin{equation}\label{def:a}
    a(t,x) := a^{X,\beta}(t,x) = 1 +\sqrt{\delta} + \frac{u_+ - \tilde{u}(t,x)}{\sqrt{\delta}}.
\end{equation}
Since the shock profile $\tilde{u}$ is monotonically decreasing from $u_-$ to $u_+$, the weight satisfies $1 \leq a \leq 1+ \sqrt{\delta} \leq \frac{3}{2}$ and
\begin{equation*}
    a_x(t,x) = -\frac{\tilde{u}_x(t,x)}{\sqrt{\delta}} > 0.
\end{equation*}
Now, we define the dynamical shift $X(t)$ as the solution of the following ODE:
\begin{equation}\label{def:X}
    \begin{aligned}
        &\dot{X}(t) = -\frac{M}{\delta}\left[\int_{\mathbb{R}_+} a^{X,\beta} \frac{p'(\tilde{\rho}^{X,\beta})}{\sigma -\tilde{u}^{X,\beta}}\tilde{\rho}^{X,\beta}_x(u-\tilde{u}^{X,\beta})\,dx + \int_{\mathbb{R}_+} a^{X,\beta} \tilde{\rho}^{X,\beta} (u - \tilde{u}^{X,\beta})\tilde{u}^{X,\beta}_x\,dx\right],\\
        &X(0)=0.
    \end{aligned}
\end{equation}
Here, $M = \frac{2(\gamma+1)}{\rho_+}$, and $\beta>0$ is a sufficiently large constant to be chosen later. The existence and Lipschitz continuity of $X(t)$ are proved in \cite{KVW23} by using standard Cauchy--Lipschitz theory.

\subsection{A priori energy estimates}

Now, we present the main proposition for proving Theorem \ref{thm:main}, which provides the a priori estimates.

\begin{proposition}\label{prop:apriori}

    For a given right-end state $(\rho_+, u_+) \in \mathbb{R}_+ \times \mathbb{R}$ satisfying $u_+<0$ and \eqref{cond:U+}, there exist positive constants $C_0, c_0, \delta_0$, and $\e_0$ such that the following holds.
    
    Let $u_- \in (u_+, 0)$ be the outflow velocity with shock strength $\delta := |u_- - u_+| < \delta_0$. Let $(\tilde{\rho}, \tilde{u})$ be the viscous-dispersive 2-shock profile connecting $(\rho_-, u_-)$ and $(\rho_+, u_+)$ with $\tilde{\rho}(0) = (\rho_- + \rho_+)/2$. Then, there exists a sufficiently large $\beta>0$, depending only on $\delta$, such that if $(\rho, u)$ is a strong solution to \eqref{eq:NSK} on $[0,T]$ satisfying
    \begin{equation*}
    \begin{aligned}
    &\rho - \tilde{\rho}^{X, \beta} \in C([0,T];H^2(\mathbb{R}_+)) \cap L^2(0,T;H^3(\mathbb{R}_+)),\\
    &u - \tilde{u}^{X, \beta} \in C([0,T];H^1(\mathbb{R}_+)) \cap L^2(0,T;H^2(\mathbb{R}_+)),
    \end{aligned} 
    \end{equation*}
    and
    \begin{equation}\label{p-assum}
    \sup_{t \in [0,T]} \Big( \| \rho - \tilde{\rho}^{X, \beta} \|_{H^2(\mathbb{R}_+)} + \| u - \tilde{u}^{X, \beta} \|_{H^1(\mathbb{R}_+)} \Big) \leq \e_0,
    \end{equation}
    then we have the following estimate for all $t \in [0, T]$:
    \begin{equation*}
    \begin{aligned}
    & \|(\rho - \tilde{\rho}^{X, \beta})(t,\cdot
    )\|_{H^2(\mathbb{R}_+)}^2 + \| (u - \tilde{u}^{X, \beta})(t,\cdot) \|_{H^1(\mathbb{R}_+)}^2 \\
    &\qquad + \int_0^t \Big( \delta |\dot{X}|^2 + G_1 + G_3 + G^S + D_{u_1} + D_{u_2} + G_{w_1} + D_{w_1} + G_{w_2} + D_{w_2} \Big)\,ds\\
    &\quad \leq C_0 \Big( \| \rho_0 - \tilde{\rho}^{0,\beta} \|_{H^2(\mathbb{R}_+)}^2 + \| u_0 - \tilde{u}^{0,\beta} \|_{H^1(\mathbb{R}_+)}^2 \Big) + C_0 e^{-c_0 \delta \beta}.
    \end{aligned} 
    \end{equation*}
    Here, the constants $C_0$ and $c_0$ are independent of $T, \beta, \delta_0,$ and $\varepsilon_0$, and
    \begin{equation}\label{terms}
    \begin{split}
        &\begin{aligned}
         G_1 = \int_{\mathbb{R}_+} a_x \frac{\sigma-\tilde{u}}{2}\gamma\tilde{\rho}^{\gamma-2}\left((\rho-\tilde{\rho})-\frac{\tilde{\rho}}{\sigma-\tilde{u}}(u-\tilde{u})\right)^2\,dx,   
        \end{aligned}\\
        &\begin{aligned}
        &G_3 = \int_{\mathbb{R}_+} a_x \frac{\sigma-\tilde{u}}{2}\tilde{\rho}|w-\tilde{w}|^2\,dx, &&
        G^S = \int_{\mathbb{R}_+} |\tilde{u}_x| |u- \tilde{u}|^2\,dx, \\
        &D_{u_1} = \mu \int_{\mathbb{R}_+} a |(u - \tilde{u})_x|^2\,dx, &&
        D_{u_2} = \mu \int_{\mathbb{R}_+} \frac{1}{\rho} |(u - \tilde{u})_{xx}|^2\,dx, \\
        &G_{w_1} = \frac{1}{\sqrt{\kappa}} \int_{\mathbb{R}_+} \frac{p'(\rho)}{\rho^{1/2}} |w - \tilde{w}|^2\,dx, &&
        D_{w_1} = \sqrt{\kappa} \int_{\mathbb{R}_+} \rho^{1/2} |(w - \tilde{w})_x|^2\,dx, \\
        &G_{w_2} = \frac{1}{\sqrt{\kappa}} \int_{\mathbb{R}_+} \frac{p'(\rho)}{\rho^{1/2}} |(w - \tilde{w})_x|^2\,dx, &&
        D_{w_2} = \sqrt{\kappa} \int_{\mathbb{R}_+} \rho^{1/2} |(w - \tilde{w})_{xx}|^2\,dx.
        \end{aligned}
    \end{split}
    \end{equation}
    Furthermore, the derivative of the shift satisfies 
    \[
    |\dot{X}(t)| \leq C_0 \big( \| (\rho - \tilde{\rho}^{X, \beta})(t,\cdot) \|_{L^\infty(\mathbb{R}_+)} + \| (u - \tilde{u}^{X, \beta})(t,\cdot) \|_{L^\infty(\mathbb{R}_+)} \big), \quad \forall t \in [0,T].
    \]
\end{proposition}

\subsection{Proof of Theorem \ref{thm:main}}
Combining the local existence in Proposition \ref{prop:loc} and the a priori estimates in Proposition \ref{prop:apriori}, wit is straightforward to prove the global existence of the solution using a standard continuation argument as in \cite{HK26,KOW25}. Since the argument is almost the same, we omit the details.

To attain the time-asymptotic convergence \eqref{asym-U}, we again follow a similar argument as in \cite{HK26,KOW25}. Precisely, we consider the following quantity
\begin{equation}\label{def:g}
	g(t):=\|(\rho-\tilde{\rho}^{X,\beta})_x\|_{H^1(\R_+)}^2+\|(u-\tilde{u}^{X,\beta})_x\|_{L^2(\R_+)}^2.
\end{equation}
Once we show that $g\in W^{1,1}(\R_+)$, we have $\lim_{t\to\infty}g(t)=0$. Then, the standard interpolation inequality implies the desired convergence toward the viscous-dispersive shock wave. Since the proof of $g\in W^{1,1}(\R_+)$ is technical, and it requires several estimates in the next section, we postpone its proof to Appendix \ref{sec:app-A}.

\medskip

Therefore, it remains to prove Proposition \ref{prop:apriori}. In the following two sections, we present $L^2$-estimates and $H^1$-estimates for the perturbations $(\rho-\tilde{\rho},u-\tilde{u},w-\tilde{w})$, which complete the proof of Proposition \ref{prop:apriori}. 

For simplicity of notation, we omit the spatial domain dependency on the function spaces, e.g., $L^2:=L^2(\R_+)$. Furthermore, we adopt the following convention for constants: $C$ and $c$ denote generic positive constants that may change from line to line, while $C_1, C_2, \ldots$ denote specific positive constants that remain fixed once introduced. All these constants are independent of $T$, $\beta$, $\delta_0$, and $\varepsilon_0$. Constants depending on additional parameters are denoted explicitly, e.g., $C_\tau$.

\section{\texorpdfstring{$L^2$}{L2}-Energy Estimates}\label{sec:L2_energy}
\setcounter{equation}{0}
In this section, we establish the $L^2$-energy estimates for the perturbation based on the method of $a$-contraction with shifts. The primary goal of this section is to prove the following lemma, which serves as the foundation for the proof of Proposition \ref{prop:apriori}.

\begin{lemma} \label{lem:L2}
	Under the hypotheses of Proposition \ref{prop:apriori}, there exist positive constants $C$ and $c$ such that for all $t \in [0, T]$, the following $L^2$-estimate holds:
    \begin{equation} \label{est-L2}
        \begin{aligned}
    &\|(\rho - \tilde{\rho}, u - \tilde{u}, w-\tilde{w}) (t,\cdot)\|_{L^2}^2 + \int_0^t \big(\delta |\dot{X}|^2 + G_1 + G_3 + G^S + D_{u_1} \big)\,ds\\
    &\quad \leq C \|(\rho - \tilde{\rho}, u - \tilde{u}, w-\tilde{w})(0,\cdot) \|_{L^2}^2 + Ce^{-c\delta \beta}+ C\varepsilon^2 \int_0^t  \|(u-\tilde{u})_{xx}(s,\cdot)\|_{L^2}^2 \, ds\\
    &\qquad + C(\delta^{3/4}+\varepsilon^2) \int_0^t  \|(w-\tilde{w})_{xx}(s,\cdot)\|_{L^2}^2 \, ds,
\end{aligned}
\end{equation}
where $G_1$, $G_3$, $G^S$, and $D_{u_1}$ are defined in \eqref{terms}.
\end{lemma}

\subsection{Weighted relative entropy estimate}
To properly handle the higher-order Korteweg term, we introduce the capillarity variable $n = \rho w$ and rewrite the NSK equations \eqref{eq:NSK} as an augmented $3 \times 3$ system of viscous conservation laws \eqref{eq:augmented-sys}:
\begin{align}\label{eq:NSK-standard}
	\pa_tU + \pa_xf(U) = \pa_x(M(U)\pa_x\nabla\eta(U)),
\end{align}
where
\begin{align*}
    U:=\begin{pmatrix}
	\rho\\m\\n
    \end{pmatrix}=\begin{pmatrix}
    \rho\\\rho u\\ \rho w
    \end{pmatrix},\quad f(U):=\begin{pmatrix}
    m\\ \frac{m^2}{\rho}+p(\rho)\\ \frac{mn}{\rho}
    \end{pmatrix},\quad M(U):=\begin{pmatrix}
    0& 0 & 0 \\ 0 & \mu & \sqrt{\kappa} \rho^{3/2}\\0 & -\sqrt{\kappa} \rho^{3/2} & 0
    \end{pmatrix}.
\end{align*}
and the mathematical entropy of the system \eqref{eq:NSK-standard} is given by
\[\eta(U):=\frac{m^2}{2\rho}+\frac{\rho^\gamma}{\gamma-1}+\frac{n^2}{2\rho}=\frac{\rho u^2}{2}+\frac{\rho^{\gamma}}{\gamma-1}+\frac{\rho w^2}{2}.\]
Similarly, the shifted viscous-dispersive shock wave \[\tilde{U}(t,x)=\tilde{U}^{X,\beta}(t,x)=\begin{pmatrix}
\tilde{\rho}^{X,\beta}(t,x)\\\tilde{m}^{X,\beta}(t,x)\\\tilde{n}^{X,\beta}(t,x)
\end{pmatrix}=\begin{pmatrix}
\tilde{\rho}^{X,\beta}(t,x)\\\tilde{\rho}^{X,\beta}(t,x)\tilde{u}^{X,\beta}(t,x)\\\tilde{\rho}^{X,\beta}(t,x)\tilde{w}^{X,\beta}(t,x)
\end{pmatrix}\]
satisfies the following system
\begin{equation}\label{eq:system_shock}
	\pa_t \tilde{U} +\pa_x f(\tilde{U}) = \pa_x(M(\tilde{U})\pa_x\nabla\eta(\tilde{U}))-\dot{X}\pa_x\tilde{U}.
\end{equation}
Now, consider the relative entropy $\eta(U|\tilde{U})$, which can be straightforwardly computed as
\begin{align*}
    \eta(U|\tilde{U})&:=\eta(U)-\eta(\widetilde{U})-\nabla\eta(U)(U-\widetilde{U})=\frac{\rho|u-\tilde{u}|^2}{2}+\frac{p(\rho|\tilde{\rho})}{\gamma-1}+\frac{\rho|w-\tilde{w}|^2}{2},
\end{align*}
and the relative flux becomes
\begin{align*}
    f(U|\tilde{U}) =f(U)-f(\widetilde{U})-D f(U)(U-\widetilde{U})=\begin{pmatrix}
	0\\ \rho|u-\tilde{u}|^2+p(\rho|\tilde{\rho})\\\rho(u-\tilde{u})(w-\tilde{w})
\end{pmatrix}.
\end{align*}
Here, $p(\rho|\tilde{\rho}):=p(\rho)-p(\tilde{\rho})-p'(\tilde{\rho})(\rho-\tilde{\rho})$ is a relative pressure. Furthermore, the entropy flux $q(U)$ satisfying the compatibility condition $\nabla q(U) = \nabla \eta(U)\nabla f(U)$ is
\[q(U) = \frac{m^3}{2\rho^2}+\frac{\gamma}{\gamma-1}\rho^{\gamma-1}m+\frac{mn^2}{2\rho^2} = u\eta(U) + up(\rho).\]
Consequently, the relative entropy flux $q(U;\tilde{U})$ is derived as
\begin{align*}
	q(U;\tilde{U})&=q(U)-q(\tilde{U})-\nabla\eta(\tilde{U})(f(U)-f(\tilde{U}))\\
	&=\frac{\rho u}{2}|u-\tilde{u}|^2+\frac{u}{\gamma-1}p(\rho|\tilde{\rho}) +\frac{\rho u}{2}|w-\tilde{w}|^2 +(p-\tilde{p})(u-\tilde{u})\\
	&=u\eta(U|\tilde{U})+(p-\tilde{p})(u-\tilde{u}).
\end{align*}

As the relative entropy is locally quadratic, we focus on estimating the relative entropy to get the desired $L^2$-estimate. We now establish the weighted relative entropy estimate for the NSK equations.

\begin{lemma} \label{lem:weighted-rel-entropy}
    Let $a$ be the weight function defined in \eqref{def:a}, $U$ be a solution to \eqref{eq:NSK-standard}, and $\tilde{U}$ be the shifted shock wave satisfying \eqref{eq:system_shock}. Then, the time evolution of the weighted relative entropy satisfies
   	\begin{equation}\label{est:weighted_relative_ent}
   		\frac{d}{dt}\int_{\mathbb{R}_+}a(t,x)\eta(U(t,x)|\tilde{U}(t,x))\,dx=\dot{X}Y +\mathcal{J}^{bad}-\mathcal{J}^{good}+P,
   	\end{equation}
    where $Y$ is defined by
    \begin{align*}
        Y&:=-\int_{\mathbb{R}_+}a_x\eta(U|\tilde{U})\,dx +\int_{\mathbb{R}_+}a\partial_x(\nabla\eta(\tilde{U}))(U-\tilde{U})\,dx\\
        &= -\int_{\mathbb{R}_+}a_x\eta(U|\tilde{U})\,dx + \int_{\mathbb{R}_+} a \gamma \tilde{\rho}^{\gamma-2}(\rho - \tilde{\rho})\tilde{\rho}_x\,dx + \int_{\mathbb{R}_+} a  \rho((u-\tilde{u})\tilde{u}_x + (w-\tilde{w})\tilde{w}_x)\,dx,
    \end{align*}
    and the bad, good, and boundary terms are given by
    \begin{align*}
        \mathcal{J}^{bad}&:=\int_{\mathbb{R}_+}a_x(p-\tilde{p})(u-\tilde{u})\,dx + \int_{\mathbb{R}_+}a_x(u-\tilde{u})\frac{p(\rho|\tilde{\rho})}{\gamma-1}\,dx\\
        &\quad +\frac{1}{2}\int_{\mathbb{R}_+}a_x\left(\rho(u-\tilde{u})^3-(\sigma-\tilde{u})(\rho-\tilde{\rho})(u-\tilde{u})^2\right)\,dx\\
        &\quad +\frac{1}{2}\int_{\mathbb{R}_+}a_x\left(\rho(u-\tilde{u})(w-\tilde{w})^2-(\sigma-\tilde{u})(\rho-\tilde{\rho})(w-\tilde{w})^2\right)\,dx\\
        &\quad -\int_{\mathbb{R}_+}a\tilde{u}_x(\rho|u-\tilde{u}|^2+p(\rho|\tilde{\rho}))\,dx-\int_{\mathbb{R}_+}a\tilde{w}_x\rho(u-\tilde{u})(w-\tilde{w})\,dx\\
        &\quad -\mu\int_{\mathbb{R}_+}a_x(u-\tilde{u})(u-\tilde{u})_x\,dx-\sqrt{\kappa}\int_{\mathbb{R}_+}a_x(u-\tilde{u})\rho^{3/2}(w-\tilde{w})_x\,dx\\
        &\quad +\sqrt{\kappa}\int_{\mathbb{R}_+}a_x(w-\tilde{w})\rho^{3/2}(u-\tilde{u})_x\,dx\\
        &\quad -\sqrt{\kappa}\int_{\mathbb{R}_+}a_x(\rho^{3/2}-\tilde{\rho}^{3/2})((u-\tilde{u})\tilde{w}_x-(w-\tilde{w})\tilde{u}_x)\,dx\\
        &\quad -\sqrt{\kappa}\int_{\mathbb{R}_+}a(\rho^{3/2}-\tilde{\rho}^{3/2})((u-\tilde{u})_x\tilde{w}_x-(w-\tilde{w})_x\tilde{u}_x)\,dx\\
        &\quad -\int_{\mathbb{R}_+}a\frac{\rho-\tilde{\rho}}{\tilde{\rho}}\left[(u-\tilde{u})\left(\mu\tilde{u}_{xx}+(\sqrt{\kappa}\tilde{\rho}^{3/2}\tilde{w}_x)_x\right)-(w-\tilde{w})(\sqrt{\kappa}\tilde{\rho}^{3/2}\tilde{u}_x)_x\right]\,dx, 
    \end{align*}
    and
    \begin{align*}
         \mathcal{J}^{good}&:=\int_{\mathbb{R}_+}a_x(\sigma-\tilde{u})\left(\tilde{\rho}\frac{|u-\tilde{u}|^2}{2}+\frac{p(\rho|\tilde{\rho})}{\gamma-1} +\tilde{\rho}\frac{|w-\tilde{w}|^2}{2} \right)\,dx +\mu\int_{\mathbb{R}_+}a|(u-\tilde{u})_x|^2\,dx, \\
        P 
         &:=\left.\left[ a u\Bigl( \rho  \frac{(u-\tilde{u})^2}{2}+ \frac{p(\rho | \tilde{\rho})}{\gamma-1} + \rho  \frac{(w-\tilde{w})^2}{2} \Bigr) \right]\right|_{x=0}\\
         &\quad + \left.\left[ a \Bigl(  (p-\tilde{p}) (u-\tilde{u}) - \mu (u-\tilde{u}) (u-\tilde{u})_x \Bigr)\right]\right|_{x=0}\\
         &\quad \left. -\sqrt{\kappa} \left[ a \Bigl(   (u-\tilde{u})\rho^{3/2}(w-\tilde{w})_x - (w-\tilde{w})\rho^{3/2}(u-\tilde{u})_x \Bigr) \right]\right|_{x=0} \\
         &\quad \left. - \sqrt{\kappa} \left[ a \Bigl((u-\tilde{u})(\rho^{3/2}-\tilde{\rho}^{3/2})\tilde{w}_x -(w-\tilde{w})(\rho^{3/2}-\tilde{\rho}^{3/2})\tilde{u}_x \Bigr)\right]\right|_{x=0}.
    \end{align*}
\end{lemma}

\begin{remark}
    We emphasize that the introduction of the weight function $a(t,x)$ is crucial in our analysis. Since $a_x > 0$ and $\sigma - \tilde{u} > 0$ for a weak 2-shock, the weight generates strictly positive localized terms in $\mathcal{J}^{good}$ weighted by $a_x(\sigma-\tilde{u})$. These terms are essential to absorb the higher-order nonlinear couplings in $\mathcal{J}^{bad}$.
\end{remark}

\begin{proof}
    Using the standard relative entropy estimate \cite{D96,D79,KV21} for the system \eqref{eq:NSK-standard}, the time evolution of the weighted relative entropy satisfies:
    \begin{align}
    	\begin{aligned}\label{est:rel-ent}
        &\frac{d}{dt}\int_{\mathbb{R}_+}a\eta(U|\tilde{U})\,dx
        =\dot{X}\left(-\int_{\mathbb{R}_+}a_x\eta(U|\tilde{U})\,dx+\int_{\mathbb{R}_+}a\partial_x(\nabla\eta(\tilde{U}))(U-\tilde{U})\,dx\right)\\
        &\quad -\sigma\int_{\mathbb{R}_+}a_x\eta(U|\tilde{U})\,dx -\int_{\mathbb{R}_+}a\partial_xq(U;\tilde{U})\,dx-\int_{\mathbb{R}_+}a\partial_x\nabla\eta(\tilde{U})f(U|\tilde{U})\,dx\\
        &\quad + \int_{\mathbb{R}_+}a(\nabla\eta(U)-\nabla\eta(\tilde{U}))\partial_x\Big(M(U)\partial_x(\nabla\eta(U)-\nabla\eta(\tilde{U}))\Big)\,dx\\
        &\quad +\int_{\mathbb{R}_+}a(\nabla\eta(U)-\nabla\eta(\tilde{U}))\partial_x\Big((M(U)-M(\tilde{U}))\partial_x\nabla\eta(\tilde{U})\Big)\,dx\\
        &\quad + \int_{\mathbb{R}_+}a(\nabla\eta)(U|\tilde{U})\partial_x\Big(M(\tilde{U})\partial_x\nabla\eta(\tilde{U})\Big)\,dx \\
        &=: \dot{X}Y -\sigma\int_{\mathbb{R}_+}a_x\eta(U|\tilde{U})\,dx + I_1 + I_2 + I_3 + I_4 + I_5.
        \end{aligned}
    \end{align}	

    Notice that the gradient of the entropy and relative entropy are
    \[(\nabla\eta)(U) = \begin{pmatrix}
    -\frac{u^2}{2}+\frac{\gamma}{\gamma-1}\rho^{\gamma-1}-\frac{w^2}{2}\\u\\w
    \end{pmatrix}, \quad 
    (\nabla\eta)(U|\tilde{U}) = \begin{pmatrix}
    *\\-\frac{(\rho-\tilde{\rho})}{\tilde{\rho}}(u-\tilde{u})\\-\frac{(\rho-\tilde{\rho})}{\tilde{\rho}}(w-\tilde{w})
    \end{pmatrix}.\]
    Using integration by parts, each term from $I_1$ to $I_5$ can be computed as
    \begin{align*}
        I_1 &= \int_{\mathbb{R}_+}a_xq(U;\tilde{U})\,dx +a(t,0)q(U;\tilde{U})(t,0)\\
        I_2 &= -\int_{\mathbb{R}_+} \Big[ a\tilde{u}_x \bigl(\rho|u-\tilde{u}|^2+p(\rho|\tilde{\rho}) \bigr)+a\tilde{w}_x\rho(u-\tilde{u})(w-\tilde{w}) \Big]\,dx,\\
        I_3 &=-\mu \int_{\mathbb{R}_+}a|(u-\tilde{u})_x|^2\,dx -\mu\int_{\mathbb{R}_+}a_x(u-\tilde{u})(u-\tilde{u})_x\,dx\\
        &\quad -\sqrt{\kappa}\int_{\mathbb{R}_+}a_x(u-\tilde{u})\rho^{3/2}(w-\tilde{w})_x\,dx+\sqrt{\kappa}\int_{\mathbb{R}_+}a_x(w-\tilde{w})\rho^{3/2}(u-\tilde{u})_x\,dx\\
        &\quad + \left. \left[ -\mu a(u-\tilde{u})(u-\tilde{u})_x-\sqrt{\kappa}a(u-\tilde{u})\rho^{3/2}(w-\tilde{w})_x +\sqrt{\kappa}a(w-\tilde{w})\rho^{3/2}(u-\tilde{u})_x \right] \right|_{x=0},\\
        I_4 &=-\sqrt{\kappa}\int_{\mathbb{R}_+}a_x(\rho^{3/2}-\tilde{\rho}^{3/2})((u-\tilde{u})\tilde{w}_x-(w-\tilde{w})\tilde{u}_x)\,dx\\
        &\quad -\sqrt{\kappa}\int_{\mathbb{R}_+}a(\rho^{3/2}-\tilde{\rho}^{3/2})\left((u-\tilde{u})_x\tilde{w}_x-(w-\tilde{w})_x\tilde{u}_x\right)\,dx\\
        &\quad + \left. \left[-\sqrt{\kappa}a(u-\tilde{u})(\rho^{3/2}-\tilde{\rho}^{3/2})\tilde{w}_x +\sqrt{\kappa}a(w-\tilde{w})(\rho^{3/2}-\tilde{\rho}^{3/2})\tilde{u}_x \right] \right|_{x=0}, \\
        I_5 &= -\int_{\mathbb{R}_+}a\frac{\rho-\tilde{\rho}}{\tilde{\rho}}\left[(u-\tilde{u})\left(\mu \tilde{u}_{xx}+(\sqrt{\kappa}\tilde{\rho}^{3/2}\tilde{w}_x)_x\right)-(w-\tilde{w})(\sqrt{\kappa}\tilde{\rho}^{3/2}\tilde{u}_x)_x\right]\,dx.
    \end{align*}
    
    Notice that 
    \[-\sigma\int_{\R_+}a_x\eta(U|\tilde{U})\,dx +I_1=\int_{\R_+}a_x(u-\sigma)\eta(U|\tilde{U})\,dx+\int_{\R_+}a_x(p-\tilde{p})(u-\tilde{u})\,dx.\]
	Thus, by decomposing $u-\sigma = (u-\tilde{u}) - (\sigma-\tilde{u})$ and $\rho = \tilde{\rho} + (\rho-\tilde{\rho})$, the leading parts $-(\sigma-\tilde{u})$ and $\tilde{\rho}$ generate the good terms in $\mathcal{J}^{good}$, while the perturbations $(u-\tilde{u})$ and $(\rho-\tilde{\rho})$ are absorbed as errors into $\mathcal{J}^{bad}$.

Finally, substituting the computations for $I_1$ through $I_5$ into \eqref{est:rel-ent}, we obtain the desired identity.

\end{proof}

Among the terms in $\mathcal{J}^{bad}$, the term
\[
    \int_{\mathbb{R}_+} a_x (p-\tilde{p})(u-\tilde{u}) \,dx
\]
is the worst term since it couples the velocity perturbation with the density perturbation. To control this term, we exploit one of the term $-\int_{\mathbb{R}_+} a_x (\sigma-\tilde{u}) \frac{p(\rho|\tilde{\rho})}{\gamma-1} \,dx$ in $J^{good}$. 

\begin{lemma}
    Under the assumptions of Proposition \ref{prop:apriori}, we have
	\begin{align*}
		&-\int_{\R_+}a_x(\sigma-\tilde{u})\frac{p(\rho|\tilde{\rho})}{\gamma-1}\,dx +\int_{\R_+}a_x(p-\tilde{p})(u-\tilde{u})\,dx\\
		&\le -\int_{\R_+}a_x\frac{\sigma-\tilde{u}}{2}\gamma\tilde{\rho}^{\gamma-2}\left((\rho-\tilde{\rho})-\frac{\tilde{\rho}}{\sigma-\tilde{u}}(u-\tilde{u})\right)^2\,dx+\int_{\R_+}a_x\frac{\gamma\tilde{\rho}^{\gamma}}{2(\sigma-\tilde{u})}(u-\tilde{u})^2\,dx\\
		&\quad +C\int_{\R_+}a_x|\rho-\tilde{\rho}|^2|u-\tilde{u}|\,dx +C\int_{\R_+} (\sigma-\tilde{u})a_x|\rho-\tilde{\rho}|^3 \, dx.
	\end{align*}
\end{lemma}

\begin{proof}
	Since the proof is a direct application of the Taylor expansion as in \cite{KOW25}, we omit the proof.
\end{proof}

Using the above lemma, we can estimate the right-hand side of \eqref{est:weighted_relative_ent} by the shift generation term $Y$, the bad term $B$, the good term $G$, and the boundary term $P$.

\begin{lemma}\label{lem:XBG} 
Under the assumptions of Proposition \ref{prop:apriori}, we have
\begin{equation}\label{eq:decomp_XBG}
    \frac{d}{dt}\int_{\mathbb{R}_+} a \eta(U|\widetilde{U}) \, dx \le \dot{X}Y + B - G + P,
\end{equation}
where the shift generating term $Y$ is decomposed as $Y := \sum_{i=1}^4 Y_i$, with
\begin{align*}
Y_1 &:= \int_{\mathbb{R}_+} a \frac{p'(\tilde{\rho})}{\sigma-\tilde{u}} \tilde{\rho}_x (u-\tilde{u}) \, dx + \int_{\mathbb{R}_+} a \tilde{\rho} (u-\tilde{u}) \tilde{u}_x \, dx,\\
Y_2 &:=  - \int_{\mathbb{R}_+} a_x  \frac{\rho (u- \tilde{u})^2}{2} \, dx - \int_{\mathbb{R}_+} a_x \frac{p(\rho|\tilde{\rho})}{\gamma-1} \, dx,\\
Y_3 &:= \int_{\mathbb{R}_+} a \frac{p'(\tilde{\rho})}{\tilde{\rho}} \left( (\rho-\tilde{\rho})-\frac{\tilde{\rho}}{\sigma-\tilde{u}}(u-\tilde{u}) \right)\tilde{\rho}_x \, dx + \int_{\mathbb{R}_+} a (\rho-\tilde{\rho})(u-\tilde{u}) \tilde{u}_x \, dx,\\
Y_4 &:=- \int_{\mathbb{R}_+} a_x  \frac{\rho(w-\tilde{w})^2}{2} \, dx  +  \int_{\mathbb{R}_+} a \rho 
(w-\tilde{w})\tilde{w}_x\, dx.
\end{align*}

Moreover, the bad term $B$ is collected by $B := \sum_{i=1}^{11} B_i$, where
\begin{align*}
	B_1&:=\int_{\mathbb{R}_+} a_x\frac{\gamma\tilde{\rho}^{\gamma}}{2(\sigma-\tilde{u})}(u-\tilde{u})^2\,dx, \quad B_2:=-\int_{\mathbb{R}_+}a\tilde{u}_x\rho(u-\tilde{u})^2\,dx,\\
	B_3&:=-\int_{\mathbb{R}_+}a\tilde{u}_x\frac{\gamma(\gamma-1)}{2}\tilde{\rho}^{\gamma-2}(\rho-\tilde{\rho})^2\,dx, \quad B_4:=-\mu\int_{\mathbb{R}_+}a_x(u-\tilde{u})(u-\tilde{u})_x\,dx,\\
	B_5&:=-\int_{\mathbb{R}_+}a\tilde{w}_x\rho(u-\tilde{u})(w-\tilde{w})\,dx, \quad B_6:=-\int_{\mathbb{R}_+}a_x\rho^{3/2}(u-\tilde{u})(w-\tilde{w})_x\,dx,\\
	B_7&:=-\int_{\mathbb{R}_+}a_x\rho^{3/2}(u-\tilde{u})_x(w-\tilde{w})\,dx,\\
	B_8&:=-\sqrt{\kappa}\int_{\mathbb{R}_+}a_x(\rho^{3/2}-\tilde{\rho}^{3/2})((u-\tilde{u})\tilde{w}_x-(w-\tilde{w})\tilde{u}_x)\,dx,\\
	B_9&:=-\sqrt{\kappa}\int_{\mathbb{R}_+}a(\rho^{3/2}-\tilde{\rho}^{3/2})((u-\tilde{u})_x\tilde{w}_x-(w-\tilde{w})_x\tilde{u}_x)\,dx,\\
	B_{10}&:=-\int_{\mathbb{R}_+}a\frac{\rho-\tilde{\rho}}{\tilde{\rho}}\left[(u-\tilde{u})\left(\mu\tilde{u}_{xx}+(\sqrt{\kappa}\tilde{\rho}^{3/2}\tilde{w}_x)_x\right)-(w-\tilde{w})(\sqrt{\kappa}\tilde{\rho}^{3/2}\tilde{u}_x)_x\right]\,dx,\\
	B_{11}&:=C\int_{\mathbb{R}_+}a_x\left(|\rho-\tilde{\rho}|+|u-\tilde{u}|+|w-\tilde{w}|\right)^3\,dx,
\end{align*}
and the good term $G$ consists of $G := G_1 + G_2 + G_3 + D_{u_1}$, where
\begin{align*}
	G_1 &= \int_{\mathbb{R}_+}a_x\frac{\sigma-\tilde{u}}{2}\gamma\tilde{\rho}^{\gamma-2}\Bigl[(\rho-\tilde{\rho})-\frac{\tilde{\rho}}{\sigma-\tilde{u}}(u-\tilde{u})\Bigr]^2\,dx,\\
	G_2&=\int_{\mathbb{R}_+}a_x \frac{\sigma-\tilde{u}}{2}\tilde{\rho}|u-\tilde{u}|^2\,dx,\quad G_3=\int_{\mathbb{R}_+}a_x\frac{\sigma-\tilde{u}}{2}\tilde{\rho}|w-\tilde{w}|^2\,dx,\\
	D_{u_1}&=\mu\int_{\mathbb{R}_+}a|(u-\tilde{u})_x|^2\,dx.
\end{align*}
\end{lemma}

\subsection{Estimates for \texorpdfstring{$\dot{X}Y+B-G$}{X-dot Y + B - G}}

In this part, we derive the estimate of the term $\dot{X}Y+B-G$ in \eqref{eq:decomp_XBG}.

\begin{lemma} \label{lem:lead}
Under the assumptions of Proposition \ref{prop:apriori}, there exists a constant $C>0$ such that, for all $t \in [0,T]$,
 \begin{equation*}
        \begin{aligned}
            \dot{X}Y + B - G &\leq -\frac{\delta}{4M}|\dot{X}|^2   -\frac{1}{2}G_1- \frac{1}{2}G_3 -\frac{\gamma+1}{10}\rho_+G^S  - \frac{1}{4}D_{u_1}+C\delta^{3/4}\|(w-\tilde{w})_x\|_{L^2}^2.
        \end{aligned}
    \end{equation*}
\end{lemma}

\begin{proof}
   
We introduce new variables $y$ and $f$:
\begin{equation}\label{def:y_and_f}
	y(t,x):=\frac{u_--\tilde{u}(x-\sigma t-X(t)-\beta)}{\delta}, \quad f(t,y):=(u-\tilde{u})\circ y^{-1}.
\end{equation}
Since $\tilde{u}_x<0$, the mapping $x\mapsto y(x,t)$ is well-defined and satisfies
\[
\frac{dy}{dx} = -\frac{\tilde{u}_x(x-\sigma t - X(t) -\beta)}{\delta} >0,
\qquad
\lim_{x\to 0} y = y_0,
\qquad
\lim_{x\to +\infty} y = 1,
\]
where
\[
y_0(t):= \frac{u_--\tilde{u}(-\sigma t - X(t) - \beta)}{\delta} >0.
\]
Moreover, by the definition of the shift in \eqref{def:X} and the a priori assumption \eqref{p-assum}, we have $|\dot{X}(t)|\le C\varepsilon$, and hence $|X(t)|\le C\varepsilon t$ for all $0\le t\le T$. In particular, for sufficiently small $\varepsilon>0$, we have
\[
|X(t)|\le \frac{\sigma}{2}t,
\qquad
\forall  t\le T,
\]
which implies
\begin{equation}\label{small-X}
-\sigma t - X(t) - \beta \le -\frac{\sigma}{2}t - \beta < 0,
\qquad
\forall  t\le T.
\end{equation}
Consequently, we have
\[
y_0(t)\le Ce^{-c\delta|\sigma t + X(t) + \beta|} \le Ce^{-c\delta\beta}.
\]
In particular, by choosing $\beta$ sufficiently large, we may assume $y_0<\frac16$ for any $t\in [0,T]$. Now, we split the terms as
\begin{align*}
	\dot{X}Y+B-G = \dot{X}Y + (B_1-G_2) +\sum_{i=2}^{11}B_{i}-G_1-G_3-D_{u_1}.
\end{align*}
In the following, we estimate each term on the right-hand side.

\vspace{2mm}
\noindent\textbf{Estimates for $B_1-G_2, B_2, \ldots, B_{11}$.} 
Since the terms $B_1-G_2$, $B_2$, $B_3$, and $B_4$ are the same as in \cite[Lemma 4.5]{KOW25}, we borrow their estimates as follows:
\begin{align*}
	&B_1-G_2\le C\delta\sqrt{\delta}\int_{y_0}^1f^2\,dy,\quad B_2\le \delta(\rho_++C\varepsilon+C\sqrt{\delta})\int_{y_0}^1f^2\,dy,\\
	&B_3\le C\delta^{1/4}G_1 +\left(\frac{\gamma-1}{2}\rho_++C\delta^{1/4}\right)\delta\int_{y_0}^1f^2\,dy,\quad B_4\le \frac{1}{40}D_{u_1}+C\delta^2\int_{y_0}^1f^2\,dy.
\end{align*}
Next, we use $|\tilde{w}_x|\le C(|\tilde{\rho}_{xx}|+|\tilde{\rho}_x|^2)\le C\delta|\tilde{\rho}_x|\le C\delta|\tilde{u}_x|$ to obtain
\begin{align*}
	B_5&\le C\delta\int_{\R_+}|\tilde{u}_x|\rho|u-\tilde{u}||w-\tilde{w}|\,dx\le C\delta\int_{\R_+}\rho|\tilde{u}_x||u-\tilde{u}|^2\,dx +C\delta\int_{\R_+}\rho|\tilde{u}_x||w-\tilde{w}|^2\,dx\\
	&\le C\delta^2\int_{y_0}^1f^2\,dy + C\delta\sqrt{\delta}\int_{\R_+}a_x|w-\tilde{w}|^2\,dx\le C\delta^2\int_{y_0}^1f^2\,dy + C\delta\sqrt{\delta}G_3.
\end{align*}
For $B_6$, we use Young's inequality and $|a_x|\le \frac{|\tilde{u}_x|}{\sqrt{\delta}}\le \delta^{3/2}$ to obtain
\begin{align*}
    B_6 &\le C\int_{\R_+}(a_x)^{3/2}|u-\tilde{u}|^2\,dx +C\int_{\R_+}(a_x)^{1/2}|(w-\tilde{w})_x|^2\,dx \\
    &\le C\delta^{5/4}\int_{y_0}^1f^2\,dy+C\delta^{3/4}\|(w-\tilde{w})_x\|_{L^2}^2.
\end{align*}
The term $B_7$ can be estimated as
\[B_7\le \frac{1}{40}D_{u_1} +C\int_{\R_+}|a_x|^2(w-\tilde{w})^2\,dx\le \frac{1}{40}D_{u_1} +C\delta^{3/2}G_3.\]
For $B_8$, we use $|\rho^{3/2}-\tilde{\rho}^{3/2}|\le C|\rho-\tilde{\rho}|$ and $|\tilde{w}_x|\le C\delta|\tilde{u}_x|$ to get
\begin{align*}
	B_8
	&\le C\int_{\R_+}a_x\left|(\rho-\tilde{\rho})-\frac{\tilde{\rho}}{\sigma-\tilde{u}}(u-\tilde{u})\right||u-\tilde{u}||\tilde{w}_x|\,dx\\
	&\quad + C\int_{\R_+}a_x\left|(\rho-\tilde{\rho})-\frac{\tilde{\rho}}{\sigma-\tilde{u}}(u-\tilde{u})\right||w-\tilde{w}||\tilde{u}_x|\,dx\\
	&\quad + C\int_{\R_+}a_x|u-\tilde{u}|^2|\tilde{w}_x|\,dx +C\int_{\R_+}a_x|u-\tilde{u}||w-\tilde{w}||\tilde{u}_x|\,dx\\
	&\le C\delta\int_{\R_+}a_x\left|(\rho-\tilde{\rho})-\frac{\tilde{\rho}}{\sigma-\tilde{u}}(u-\tilde{u})\right|^2\,dx+C\delta\int_{\R_+}a_x|\tilde{u}_x|^2|u-\tilde{u}|^2\,dx\\
	&\quad + C\delta^2\int_{\R_+}a_x\left|(\rho-\tilde{\rho})-\frac{\tilde{\rho}}{\sigma-\tilde{u}}(u-\tilde{u})\right|^2\,dx+C\delta^2\int_{\R_+}a_x|w-\tilde{w}|^2\,dx\\
	&\quad + C\delta^{3}\int_{\R_+}a_x|u-\tilde{u}|^2\,dx + C\delta^2\int_{\R_+}a_x|u-\tilde{u}|^2\,dx +C\delta^2\int_{\R_+}a_x|w-\tilde{w}|^2\,dx\\
	&\le C\delta G_1 + C\delta^{5/2}\int_{y_0}^1f^2\,dy+C\delta^2G_3.
\end{align*}
Next, we estimate $B_9$. Using similar argument as in $B_8$, we have
\begin{align*}
	B_9
	&\le C\int_{\R_+}\left|(\rho-\tilde{\rho})-\frac{\tilde{\rho}}{\sigma-\tilde{u}}(u-\tilde{u})\right||(u-\tilde{u})_x||\tilde{w}_x|\,dx\\
	&\quad + C\int_{\R_+}\left|(\rho-\tilde{\rho})-\frac{\tilde{\rho}}{\sigma-\tilde{u}}(u-\tilde{u})\right||(w-\tilde{w})_x||\tilde{u}_x|\,dx\\
	&\quad + C\int_{\R_+}|u-\tilde{u}||(u-\tilde{u})_x||\tilde{w}_x|\,dx +C\int_{\R_+}|u-\tilde{u}||(w-\tilde{w})_x||\tilde{u}_x|\,dx\\
	&\le C\int_{\R_+}|\tilde{w}_x|\left|(\rho-\tilde{\rho})-\frac{\tilde{\rho}}{\sigma-\tilde{u}}(u-\tilde{u})\right|^2\,d x+\int_{\R_+}|\tilde{w}_x||(u-\tilde{u})_x|^2\,d x\\
	&\quad + C\int_{\R_+}|\tilde{u}_x|\left|(\rho-\tilde{\rho})-\frac{\tilde{\rho}}{\sigma-\tilde{u}}(u-\tilde{u})\right|^2\,d x+\int_{\R_+}|\tilde{u}_x||(w-\tilde{w})_x|^2\,d x\\
	&\quad + C\int_{\R_+}|\tilde{w}_x||u-\tilde{u}|^2\,dx +\int_{\R_+}|\tilde{w}_x||(u-\tilde{u})_x|^2\,dx \\
    &\quad +C\int_{\R_+}|\tilde{u}_x|^{3/2}|u-\tilde{u}|^2\,dx +C\int_{\R_+}|\tilde{u}_x|^{1/2}|(w-\tilde{w})_x|^2\,dx\\
	&\le C\delta^{1/2}G_1 + C\delta^2D_{u_1}+C\delta^{2}\int_{y_0}^1f^2\,dy + C\delta \|(w-\tilde{w})_x\|_{L^2}^2.
\end{align*}
	We use $|\tilde{u}_{xx}|\le C\delta |\tilde{u}_x|$ and $|\tilde{w}_{xx}|\le C(|\tilde{\rho}_x|^2+|\tilde{\rho}_x||\tilde{\rho}_{xx}|+|\tilde{\rho}_{xxx}|)\le C\delta |\tilde{u}_x|$ to estimate $B_{10}$ as
\begin{align*}
	B_{10}
	&\le C\delta\int_{\R_+}|\tilde{u}_x||\rho-\tilde{\rho}|(|u-\tilde{u}|+|w-\tilde{w}|)\,dx\\
	&\le C\delta^{3/2}\int_{\R_+}a_x\left|(\rho-\tilde{\rho})-\frac{\tilde{\rho}}{\sigma-\tilde{u}}(u-\tilde{u})\right|^2\,dx + C\delta\int_{\R_+}|\tilde{u}_x||u-\tilde{u}|^2\,dx+C\delta\int_{\R_+}|\tilde{u}_x||w-\tilde{w}|^2\,dx\\
	&\le C\delta^{3/2}G_1 + C\delta^2\int_{y_0}^1f^2\,dy+C\delta^{3/2}G_3.
\end{align*}
Finally, we follow the estimate in \cite[Lemma 4.5]{KOW25} to get
\begin{align*}
	B_{11}&\le C\int_{\R_+}a_x\left|\rho-\tilde{\rho}-\frac{\tilde{\rho}}{\sigma-\tilde{u}}(u-\tilde{u})\right|^3\,dx + C\int_{\R_+}a_x|u-\tilde{u}|^3\,dx +C\int_{\R_+}a_x|w-\tilde{w}|^3\,dx\\
	&\le C\varepsilon G_1 + C\varepsilon D_{u_1}+C\varepsilon\delta\int_{y_0}^1f^2\,dy+C\e G_3.
\end{align*}
Collecting all the estimates above, we get
\begin{equation} \label{est-B}
\begin{aligned}
    B -G_2 &\le \delta\left(\frac{\gamma+1}{2}\rho_++C\varepsilon+C\delta^{1/4}\right)\int_{y_0}^1f^2\,dy\\
    &\quad +C(\delta^{1/4}+\varepsilon)G_1 +\frac{1}{10}D_{u_1}+C(\e + \delta^{3/2})G_3
    +C\delta^{3/4}\|(w-\tilde{w})_x\|_{L^2}^2.
\end{aligned}
\end{equation}

\vspace{2mm}
\noindent \textbf{Estimate of $D_{u_1}$.}
Next, we estimate $D_{u_1}$. Using the change of variable $x\mapsto y$, $D_{u_1}$ can be represented as
\begin{equation}\label{est:Du1}
D_{u_1}= \mu\int_{\R_+}a|(u-\tilde{u})_x|^2\,dx\ge\mu\int_{y_0}^1|\pa_yf|^2\frac{dy}{dx}dy.
\end{equation}
In the following lemma, we estimate the Jacobian $\frac{dy}{dx}$.

\begin{lemma} \label{lem:jacobian}
    There exist positive constants $\delta_0$ and $C$ such that for any shock strength $\delta \in (0, \delta_0)$, the following estimate holds:
    \begin{equation*}
        \left| \mu \frac{1}{y(1-y)}\frac{dy}{dx} - \frac{(\gamma+1)\rho_+ \delta}{2} \right| \le C\delta^2.
    \end{equation*}
\end{lemma}

\begin{proof}
    Although the Jacobian estimate is well-established in \cite{KOW25}, we present the details on treating the capillarity term for the readers' convenience. From the definition of $y$ in \eqref{def:y_and_f}, we have
    \begin{align*}
        \mu\frac{1}{y(1-y)}\frac{d y}{d x} = \mu\left(\frac{\tilde{u}'}{\tilde{u} - u_-} - \frac{\tilde{u}'}{\tilde{u} - u_+}\right).
    \end{align*}
    On the other hand, integrating \eqref{eq:VCshock}$_1$ over $(-\infty,\xi)$ and $(\xi,\infty)$, we first observe that
    \begin{equation}\label{eq:mass-flux}
        \tilde{\rho}(\tilde{u} - \sigma) = \rho_\pm(u_\pm - \sigma) =: m.
    \end{equation}
    Using \eqref{eq:mass-flux}, we integrate \eqref{eq:VCshock}$_2$ to obtain
    \begin{equation}\label{eq:int-mom}
        \mu \tilde{u}' + \kappa K(\tilde{\rho}) = m(\tilde{u} - u_\pm) + \tilde{p} - p_\pm.
    \end{equation}
    
    Dividing \eqref{eq:int-mom} by $(\tilde{u} - u_\pm)$ and taking the difference between the two far-field cases yields
    \begin{equation}\label{est:diff_far_field}
        \mu \left( \frac{\tilde{u}'}{\tilde{u} - u_-} - \frac{\tilde{u}'}{\tilde{u} - u_+} \right) = \left( \frac{\tilde{p} - p_-}{\tilde{u} - u_-} - \frac{\tilde{p} - p_+}{\tilde{u} - u_+} \right) - \kappa K(\tilde{\rho}) \left( \frac{1}{\tilde{u} - u_-} - \frac{1}{\tilde{u} - u_+} \right).
    \end{equation}
    Thanks to \cite[Lemma 4.6]{KOW25}, the pressure discrepancy for weak shocks provides the leading-order term:
    \begin{equation}\label{est:Jac_main}
        \frac{\tilde{p} - p_-}{\tilde{u} - u_-} - \frac{\tilde{p} - p_+}{\tilde{u} - u_+} = \frac{\gamma + 1}{2}\rho_+ \delta + \mathcal{O}(\delta^2).
    \end{equation}
    Regarding the capillarity contribution, Lemma~\ref{lem:shock-profile} implies that
    \[|K(\tilde{\rho})| \le C \delta |\tilde{\rho}_x|\leq C \delta |\tilde{u}_x|.\] 
    Therefore, the capillarity term is bounded as
    \begin{equation}\label{est:cap_error}
        \kappa K(\tilde{\rho}) \left( \frac{1}{\tilde{u} - u_-} - \frac{1}{\tilde{u} - u_+} \right) = \mathcal{O}(\delta)\mu \left( \frac{\tilde{u}'}{\tilde{u} - u_-} - \frac{\tilde{u}'}{\tilde{u} - u_+} \right).
    \end{equation}
    
    
    Substituting \eqref{est:Jac_main} and \eqref{est:cap_error} into \eqref{est:diff_far_field}, we obtain the desired estimate.
\end{proof}

Therefore, using \eqref{est:Du1} and Lemma \ref{lem:jacobian}, $D_{u_1}$ can be estimated as
\begin{align}
\begin{aligned}\label{est:Du1_2}
	D_{u_1} &\ge \delta\left(\frac{\gamma+1}{2}\rho_+ - C\delta\right) \int_{y_0}^1 y(1-y) |\partial_y f|^2 \,dy\\
	&\ge\delta\left(\frac{\gamma+1}{2}\rho_+ - C\delta\right) \int_{y_0}^1 (y-y_0)(1-y) |\partial_y f|^2 \,dy.
\end{aligned}
\end{align}

\vspace{2mm}
\noindent\textbf{Estimate of $\dot{X}Y$.} Finally, we estimate $\dot{X}Y$. Since $\dot{X}(t)=-(M/\delta)Y_1$, we have
\begin{align*}
\dot{X} Y
&= -\frac{\delta}{M}|\dot{X}|^2+\dot{X} \sum_{i=2}^4 Y_i
\le -\frac{\delta}{2M}|\dot{X}|^2+\frac{C}{\delta} \sum_{i=2}^4 |Y_i|^2\\
&=-\frac{\delta}{4M}|\dot{X}|^2-\frac{M}{4\delta}|Y_1|^2+\frac{C}{\delta}\sum_{i=2}^4|Y_i|^2.
\end{align*}
Using \eqref{eq:VCshock}, we estimate $Y_1$ as
\begin{align*}
	Y_1&=\int_{\R_+}a\frac{p'(\tilde{\rho})}{\sigma-\tilde{u}}\tilde{\rho}_x(u-\tilde{u})\,dx +\int_{\R_+}a\tilde{\rho}(u-\tilde{u})\tilde{u}_x\,dx\\
	&=2\int_{\R_+}a\tilde{\rho}\tilde{u}_x(u-\tilde{u})\,dx+\mu\int_{\R_+}a\frac{\tilde{u}_{xx}}{\sigma-\tilde{u}}(u-\tilde{u})\,dx +\kappa\int_{\R_+}a\frac{K(\tilde{\rho})_x}{\sigma-\tilde{u}}(u-\tilde{u})\,dx.
\end{align*}
Using $|a-1|\le \sqrt{\delta}$ and the smallness of the derivatives $\tilde{u}_{xx}$ and $K(\tilde{\rho})_x$, we get
\begin{equation*}
\left|Y_1 + 2\rho_+\delta\int_{y_0}^1 f \, dy\right| \le C\delta\sqrt{\delta}\int_{y_0}^1 |f| \, dy.
\end{equation*}
Using the algebraic inequality $\frac{a^2}{2} - b^2 \leq (a-b)^2$ for any $a,b\geq 0$, we obtain
\begin{equation*}
-\frac{M}{4\delta}|Y_1|^2 \le -\frac{M\rho_+^2}{2}\delta \left(\int_{y_0}^1 f\,dy\right)^2+C\delta^2\int_{y_0}^1|f|^2\,dy.
\end{equation*}

Next, we use $p(\rho|\tilde{\rho})\le C|\rho-\tilde{\rho}|^2$, Sobolev interpolation inequality and the a priori smallness assumption \eqref{p-assum} to estimate $Y_2$ as
\begin{equation*}
\frac{C}{\delta}|Y_2|^2 \le C\sqrt{\delta} \varepsilon^2 G_1 + C\delta \varepsilon^2\int_{y_0}^1 f^2 \, dy.
\end{equation*}
Similarly, for $Y_3$, one obtains
\begin{equation*}
\frac{C}{\delta}|Y_3|^2 \le C\sqrt{\delta} G_1 + C\delta^2\varepsilon^2\int_{y_0}^1 f^2 \, dy.
\end{equation*}
Finally, using $|\tilde{w}_x| \le C \delta |\tilde{\rho}_x| \le C \delta \sqrt{\delta} a_x$ and the smallness assumption \eqref{p-assum}, $Y_4$ can be estimated as 
\begin{align*}
    \frac{C}{\delta} |Y_4|^2 &\le \frac{C}{\delta} \left(\left| \int_{\R_+} a_x  (w-\tilde{w})^2 \, dx\right|^2  + \left|\int_{\R_+} a |\tilde{w}_x| |w-\tilde{w}|\, dx \right|^2\right) \le C \sqrt{\delta} (\varepsilon^2 + \delta^2) G_3.
\end{align*}
Combining the above estimates, we have
\begin{equation} \label{est-XY}
    \begin{aligned}
       \dot{X}Y&\le -\frac{\delta}{4M}|\dot{X}|^2-\frac{M}{4\delta}|Y_1|^2+\frac{C}{\delta} \sum_{i=2}^4 |Y_i|^2 \\
       &\le -\frac{\delta}{4M}|\dot{X}|^2-\delta\left(\frac{M\rho_+^2}{2}\right)\left(\int_{y_0}^1 f \, dy\right)^2 + C \delta(\varepsilon^2+\delta) \int_{y_0}^1 f^2 \, dy + C \sqrt{\delta} G_1 + C \sqrt{\delta} (\varepsilon^2 + \delta^2) G_3.
    \end{aligned}
\end{equation}

\vspace{2mm}
\noindent \textbf{Conclusion.}
Using \eqref{est-B}, \eqref{est:Du1_2}, \eqref{est-XY} and the smallness of $\delta$ and $\e$, we estimate $\dot{X}Y+B-G$ as
\begin{align*}
    \dot{X}Y + B - G &\le \left( -\frac{\delta}{4M}|\dot{X}|^2 - \frac{1}{2}G_1 - \frac{1}{2}G_3 - \frac{1}{4}D_{u_1} \right) \\
    &\quad + \left( -\frac{M}{4\delta}|Y_1|^2 + \frac{C}{\delta} \sum_{i=2}^4 |Y_i|^2 + B  - G_2 - \frac{1}{2}G_1 - \frac{1}{2}G_3 - \frac{3}{4}D_{u_1} \right)\\
    &\le \left( -\frac{\delta}{4M}|\dot{X}|^2 - \frac{1}{2}G_1 - \frac{1}{2}G_3 - \frac{1}{4}D_{u_1} \right)\\
    &\quad +\delta\Bigg[-\frac{M\rho_+^2}{2} \left(\int_{y_0}^1 f \, dy\right)^2  +\left(\frac{\gamma+1}{2}\rho_++C\e+C\delta^{1/4}\right)\int_{y_0}^1f^2\,dy \\
    &\qquad -\frac{5}{8}\left(\frac{\gamma+1}{2}\rho_+-C\delta\right)\int_{y_0}^1(y-y_0)(1-y)|\pa_yf|^2\,dy \Bigg] +C\delta^{3/4}\|(w-\tilde{w})_x\|_{L^2}^2.
\end{align*}

On the other hand, recall the following Poincar\'e-type inequality \cite{HKKL25-JMAA} for any $f:[y_0,1]\to\R$:
\begin{equation*}
    \int_{y_0}^1 f^2 \, dy - \frac{1}{1-y_0}\left(\int_{y_0}^1 f \, dy \right)^2 \leq \frac{1}{2}\int_{y_0}^1 (y-y_0)(1-y)|\partial_y f|^2 \, dy.
\end{equation*}
This implies, with the choice of $y_0 < \frac{1}{6}$, that
\begin{align*}
	&-\frac{5}{8}\left(\frac{\gamma+1}{2}\rho_+-C\delta\right)\int_{y_0}^1(y-y_0)(1-y)|\pa_yf|^2\,dy\\
	&\le-\frac{5}{4}\left(\frac{\gamma+1}{2}\rho_+-C\delta\right)\int_{y_0}^1f^2\,dy+\frac{3}{2}\left(\frac{\gamma+1}{2}\rho_+-C\delta\right)\left(\int_{y_0}^1f\,dy\right)^2.
\end{align*}

Since we choose $M=\frac{2(\gamma+1)}{\rho_+}$, we get
\begin{align*}
	&\dot{X}Y+B-G\\&\le -\frac{\delta}{4M}|\dot{X}|^2-\frac{1}{2}G_1-\frac{1}{2}G_3-\frac{1}{4}D_{u_1}\\
	&\ -\frac{\gamma+1}{10}\rho_+\int_{y_0}^1f^2\,dy +\delta\left(-\frac{M\rho_+^2}{2}+\frac{3(\gamma+1)}{4}\rho_+\right)\left(\int_{y_0}^1f\,dy\right)^2+C\delta^{3/4}\|(w-\tilde{w})_x\|_{L^2}^2\\
	&\le -\frac{\delta}{4M}|\dot{X}|^2-\frac{1}{2}G_1-\frac{1}{2}G_3-\frac{1}{4}D_{u_1} -\frac{\gamma+1}{10}\rho_+G^S+C\delta^{3/4}\|(w-\tilde{w})_x\|_{L^2}^2,
\end{align*}
where
\[G^S:=\int_{y_0}^1 f^2\,dy=\int_{\R_+}|\tilde{u}_x||u-\tilde{u}|^2\,dx.\]

\end{proof}

\subsection{Estimate on \texorpdfstring{$P$}{P}}

In this subsection, we decompose and estimate the boundary term $P$ defined in Lemma \ref{lem:weighted-rel-entropy}. We split the boundary term into $P = P_1 + P_2$ where,
\begin{align*}
    P_1 &:= \left. \left[ a u \left(\rho  \frac{(u-\tilde{u})^2}{2} + \frac{p(\rho|\tilde{\rho})}{\gamma-1} \right) + a (p-\tilde{p}) (u-\tilde{u}) - \mu a (u-\tilde{u}) (u-\tilde{u})_x \right]\right|_{x=0}, \\
    P_2 &:= \bigg[ a u \frac{\rho (w-\tilde{w})^2}{2} - \sqrt{\kappa} a \rho^{3/2} (u-\tilde{u})(w-\tilde{w})_x + \sqrt{\kappa} a \rho^{3/2} (w-\tilde{w})(u-\tilde{u})_x \\
    &\qquad - \sqrt{\kappa} a(u-\tilde{u})(\rho^{3/2}-\tilde{\rho}^{3/2})\tilde{w}_x + \sqrt{\kappa} a(w-\tilde{w})(\rho^{3/2}-\tilde{\rho}^{3/2})\tilde{u}_x \bigg]\bigg|_{x=0}.
\end{align*}

\begin{lemma}\label{lem:bd}
Under the assumptions of Proposition \ref{prop:apriori}, there exist positive constants $C$ and $c$ such that, for all $t \in [0,T]$,
\begin{align*}
    \int_0^t P(s) \,ds &\leq C e^{-c\delta \beta} + C\varepsilon^2 \int_0^t \left( \|(u-\tilde{u})_{xx}(s,\cdot)\|_{L^2}^2 + \|(w-\tilde{w})_{xx}(s,\cdot)\|_{L^2}^2 \right) \,ds.
\end{align*}
\end{lemma}

\begin{proof}
First of all, it follows from \eqref{small-X} that for all $0\le s \leq T$, 
\[-\sigma s - X(s) - \beta \leq -\frac{\sigma}{2}s - \beta < 0,\]  
which implies
\begin{equation}\label{eq:exp_decay}
    \int_0^t e^{-c\delta|-\sigma s - X(s) - \beta|} \, ds \le C\delta^{-1}e^{-c\delta\beta}.
\end{equation}


\noindent \textbf{Estimates for $P_1$.} 
First of all, we observe that the first term is always non-positive, due to the outflow boundary condition $u(s,0)=u_-<0$:
\[\left[au\left(\rho\frac{(u-\tilde{u})^2}{2}+\frac{p(\rho|\tilde{\rho})}{\gamma-1}\right)\right]\bigg|_{x=0}\le 0.\]
Next, we use Lemma \ref{lem:shock-profile} to get
\[
    |(u-\tilde{u})(s,0)| = |u_- - \tilde{u}(-\sigma s -X(s)-\beta)| \le C \delta e^{-c\delta|-\sigma s - X(s) - \beta|}.
\]
Using the smallness assumption \eqref{p-assum} and \eqref{eq:exp_decay}, the second term is bounded by
\begin{align*}
    \int_0^t \bigl| \bigl[ a (p-\tilde{p}) (u-\tilde{u}) \bigr](s,0) \bigr| \, ds &\leq C\|p-\tilde{p}\|_{L^\infty(\mathbb{R}_+)}\int_0^t |(u-\tilde{u})(s,0)| \, ds\\
    &\leq C\delta \varepsilon \int_0^t e^{-c\delta|-\sigma s - X(s) - \beta|} \, ds \leq C\varepsilon e^{-c\delta\beta}.
\end{align*}
For the last term, we apply Sobolev interpolation inequality, Young's inequality, and \eqref{p-assum} to get
\begin{align*}
    &\left|\int_0^t \bigl[ a (u-\tilde{u}) (u-\tilde{u})_x \bigr] (s,0) \, ds\right| \\
    &\quad \leq C \int_0^t |(u-\tilde{u})(s,0)| \|(u-\tilde{u})_x\|_{L^2}^{1/2} \|(u-\tilde{u})_{xx}\|_{L^2}^{1/2} \, ds\\
    &\quad \leq C\int_0^t |(u-\tilde{u})(s,0)|^{4/3} \, ds + C\int_0^t \|(u-\tilde{u})_x\|_{L^2}^{2} \|(u-\tilde{u})_{xx}\|_{L^2}^{2} \, ds\\
    &\quad \leq C\delta^{1/3}e^{-c\delta\beta} + C\varepsilon^2\int_0^t \|(u-\tilde{u})_{xx}(s,\cdot)\|_{L^2}^2 \, ds.
\end{align*}

\noindent \textbf{Estimates for $P_2$.} 
Similar to $P_1$, the first term is always non-positive thanks to the outflow boundary condition $u(s,0)=u_-<0$. As in the estimate of $P_1$, the second term of $P_2$ is bounded by:
\begin{align*}
    &\left| \int_0^t \left[ a \rho^{3/2} (u-\tilde{u})(w-\tilde{w})_x \right](s,0) \, ds \right| \\
    &\quad \leq C \int_0^t |(u-\tilde{u})(s,0)| \|(w-\tilde{w})_x\|_{L^2}^{1/2} \|(w-\tilde{w})_{xx}\|_{L^2}^{1/2} \, ds\\
    &\quad \leq C\int_0^t |(u-\tilde{u})(s,0)|^{4/3} \, ds + C\int_0^t \|(w-\tilde{w})_x\|_{L^2}^{2} \|(w-\tilde{w})_{xx}\|_{L^2}^{2} \, ds\\
    &\quad \leq C\delta^{1/3}e^{-c\delta\beta} + C\varepsilon^2\int_0^t \|(w-\tilde{w})_{xx}(s,\cdot)\|_{L^2}^{2} \, ds.
\end{align*}
Due to the boundary condition $w(s,0)=0$, we get
\[
|(w-\tilde{w})(s,0)| = |\tilde{w}(-\sigma s -X(s)-\beta)| \le C | \tilde{\rho}_x(s,0) | \le C \delta^2 e^{-c\delta|-\sigma s - X(s) - \beta|}.
\] 
Thus, by using the same argument above, the third term of $P_2$ is bounded as 
\begin{align*}
    &\left| \int_0^t \left[ a \rho^{3/2} (w-\tilde{w})(u-\tilde{u})_x \right](s,0) \, ds \right| \\
    &\quad \le C\int_0^t |(w-\tilde{w})(s,0)|^{4/3} \, ds + C\varepsilon^2\int_0^t \|(u-\tilde{u})_{xx}\|_{L^2}^{2} \, ds \\
    &\quad \leq C\delta^{5/3}e^{-c\delta\beta} + C\varepsilon^2\int_0^t \|(u-\tilde{u})_{xx}(s,\cdot)\|_{L^2}^{2} \, ds. 
\end{align*}
Finally, using the bound $|\tilde{w}_x| \le C \delta |\tilde{\rho}_x|$, \eqref{p-assum}, and \eqref{eq:exp_decay}, the remaining terms of $P_2$ are estimated as
\begin{align*}
    \int_0^t \left| \left[ a(w-\tilde{w})(\rho^{3/2}-\tilde{\rho}^{3/2})\tilde{u}_x \right](s,0) \right| \, ds &\leq C \varepsilon \int_0^t \delta^{4}e^{-c\delta|-\sigma s - X(s) - \beta|} \, ds \leq C \varepsilon \delta^3 e^{-c\delta\beta}, \\
    \int_0^t \left| \left[ a(u-\tilde{u})(\rho^{3/2} - \tilde{\rho}^{3/2})\tilde{w}_x \right](s,0) \right| \, ds &\leq C\varepsilon \int_0^t \delta^4 e^{-c\delta|-\sigma s - X(s) - \beta|} \, ds \leq C \varepsilon \delta^3 e^{-c\delta\beta}.
\end{align*}

Collecting all the estimates above, we obtain the desired bound.
\end{proof}

\subsection{Proof of Lemma \ref{lem:L2}}

From Lemma \ref{lem:XBG} and Lemma \ref{lem:lead}, we have 
\begin{align*}
    \frac{d}{dt} \int_{\mathbb{R}_+} a \eta(U|\widetilde{U})\,dx
    &\le -\frac{\delta}{4M}|\dot{X}|^2   -\frac{1}{2}G_1- \frac{1}{2}G_3   -\frac{\gamma+1}{10}\rho_+G^S \\
    &\quad - \frac{1}{4}D_{u_1} +C\delta^{3/4}\|(w-\tilde{w})_x\|_{L^2}^2 + P.
\end{align*}
Integrating the above inequality over $[0,t]$, using $1\le a\le \frac{3}{2}$, and applying Lemma \ref{lem:bd}, we obtain
\begin{align*}
    &\int_{\mathbb{R}_+} \eta(U(t,x)|\widetilde{U}(t,x))\,dx + \int_0^t \big(\delta |\dot{X}|^2 + G_1 + G_3 + G^S + D_{u_1}  \big)\,ds\\
    &\quad \leq C \int_{\mathbb{R}_+}  \eta(U(0,x)|\widetilde{U}(0,x))\,dx  + Ce^{-c\delta \beta} + C\varepsilon^2 \int_0^t  \|(u-\tilde{u})_{xx}(s,\cdot)\|_{L^2}^2 \, ds\\
    &\qquad + C(\delta^{3/4}+\varepsilon^2) \int_0^t  \|(w-\tilde{w})_{xx}(s,\cdot)\|_{L^2}^2 \, ds.
\end{align*}
Since the relative entropy is locally quadratic, the relative entropy is equivalent to the $L^2$-norm due to the smallness assumption \eqref{p-assum}:
\[
     \int_{\mathbb{R}_+} \eta \big(U(t,x)|\widetilde{U}(t,x) \big)\,dx \sim \|(\rho - \tilde{\rho}, u - \tilde{u}, w-\tilde{w}) (t,\cdot) \|_{L^2}^2,
\]
which yields the desired $L^2$ estimate \eqref{est-L2} in Lemma \ref{lem:L2}.

\section{Higher order estimates} \label{sec:higher_order}
To complete the proof of Proposition~\ref{prop:apriori}, we need to attain the high-order estimates. For simplicity, we define the perturbations of the density, velocity and capillarity variables as
$$\phi:=\rho-\tilde{\rho},\qquad\psi:=u-\tilde{u},\qquad\omega=w-\tilde{w}.$$
Then $(\phi, \psi, \omega)$ satisfies the perturbed equations:
\begin{equation}\label{eq:perturbed-eq}
\left\{
\begin{aligned}
&\phi_t +u \phi_x +\rho \psi_x = F+\dot{X}(t) \tilde{\rho}_x,\\
&\rho\bigl(\psi_t+u \psi_x\bigr)+\frac{p'(\rho)\rho^{1/2}}{\sqrt{\kappa}}\omega =\mu \psi_{xx}+\sqrt{\kappa} (\rho^{3/2}\omega_x)_x +G+\dot{X}(t) \rho \tilde{u}_x,\\
&\rho(\omega_t +u \omega_x) = -\sqrt{\kappa} (\rho^{3/2}\psi_x)_x+ H+ \dot{X}(t) \rho \tilde{w}_x,
\end{aligned}
\right.
\end{equation}
where the error terms $F, G, H$ are given by
\begin{equation*}
\begin{aligned}
F&:= -\tilde{\rho}_x \psi-\tilde{u}_x \phi,\\
G&:= \tilde{\rho}_x \left[ \frac{p'(\tilde{\rho})}{\tilde{\rho}} \phi -\bigl(p'(\rho)-p'(\tilde{\rho})\bigr)  - \frac{p'(\rho)}{\tilde{\rho}^{1/2}} \bigl( \rho^{1/2} -\tilde{\rho}^{1/2} \bigr) \right]-\rho \tilde{u}_x \psi -\mu \tilde{u}_{xx} \frac{\phi}{\tilde{\rho}} \\
&\quad + \sqrt{\kappa}\left[ (\rho^{3/2}-\tilde{\rho}^{3/2})\tilde{w}_x \right]_x - \sqrt{\kappa}\frac{\phi}{\tilde{\rho}}(\tilde{\rho}^{3/2} \tilde{w}_x)_x, \\
H&:= -\rho \psi \tilde{w}_x - \sqrt{\kappa}\left[ (\rho^{3/2}-\tilde{\rho}^{3/2})\tilde{u}_x \right]_x + \sqrt{\kappa}\frac{\phi}{\tilde{\rho}}(\tilde{\rho}^{3/2} \tilde{u}_x)_x.
\end{aligned}
\end{equation*}

Then, the goal of this section is to prove the following $H^1$-estimate.

\begin{lemma}\label{lem:H1}
	Under the assumptions of Proposition \ref{prop:apriori}, there exist positive constants $C$ and $c$ such that, for all $t \in [0,T]$,
	\begin{align}
		\begin{aligned}\label{est:H1}
		&\|\phi(t,\cdot)\|_{L^2}^2+\|(\psi,\omega)(t,\cdot)\|_{H^1}^2 \\
		&\quad +\int_0^t \bigl(\delta|\dot{X}|^2+G_1+G_3+G^S+D_{u_1}+G_{w_1}+D_{w_1}+D_{u_2}+G_{w_2}+D_{w_2}\bigr)\,ds\\
		&\quad \le C\left(\|\phi_0\|_{L^2}^2+\|(\psi_0,\omega_0)\|_{H^1}^2\right)+C e^{-c \delta\beta}.
		\end{aligned}
	\end{align}
\end{lemma}

\begin{remark}
	It follows from the definition of $\omega$ that
	\[\omega=\sqrt{\frac{\kappa}{\rho}}\phi_x +\sqrt{\kappa}\tilde{\rho}_x\left(\frac{1}{\sqrt{\rho}}-\frac{1}{\sqrt{\tilde{\rho}}}\right).\]
	Thus, using the smallness assumption \eqref{p-assum}, we get
	\[\|\phi_x\|_{L^2}\le C(\|\omega\|_{L^2}+\|\phi\|_{L^2}),\quad \|\phi_{xx}\|_{L^2}\le C(\|\omega\|_{H^1}+\|\phi\|_{L^2}),\quad \|\omega\|_{H^1}\le C\|\phi\|_{H^2}, \]
	which implies that $\|\phi\|_{L^2}^2+\|\omega\|_{H^1}^2$ is equivalent to $\|\phi\|_{H^2}^2$. Therefore, \eqref{est:H1} is equivalent to the desired estimate in Proposition \ref{prop:apriori}.
\end{remark}

For the subsequent estimate, it is necessary to establish bounds for the error terms $F$, $G$, and $H$. The estimate for $F$ readily follows from 
$|\tilde{\rho}_x| \le C|\tilde{u}_x|$ that
\begin{equation*} 
|F| \le C|\tilde{u}_x| |(\phi, \psi)|.
\end{equation*}
Similarly, we obtain the following bounds for $G$ and $H$:
\begin{equation} \label{eq:GH-est}
|(G, H)|\le C|\tilde{u}_x||(\phi,\psi,\phi_x)| \le C|\tilde{u}_x||(\phi,\psi,\omega)|.
\end{equation}
Furthermore, we observe the following integral bound:
\begin{equation} \label{est-quadratic}
\begin{aligned}
\int_{\mathbb{R}_+} |\tilde{u}_x| |(\phi,\psi,\omega)|^2 \, dx 
&\le \int_{\mathbb{R}_+} \left[ \left( \phi-\frac{\tilde{\rho}}{\sigma-\tilde{u}}\psi\right)^2 + C |\psi|^2 + C|\omega|^2 \right] |\tilde{u}_x| \, dx \\
&\le C \sqrt{\delta} (G_1 + G_3) + G^S.
\end{aligned}
\end{equation}

The following lemma provides dissipation for $\omega$ in the $L^2$-estimate of $(\phi,\psi,\omega)$.


\begin{lemma} \label{lem:higeroder-1}
	Under the assumptions of Proposition \ref{prop:apriori}, there exist positive constants $C$ and $c$ such that, for all $t \in [0,T]$,
\begin{equation} \label{est-L2-combined}
\begin{aligned}
    &\|(\phi, \psi, \omega) (t,\cdot)\|_{L^2}^2 + \int_0^t \big(\delta |\dot{X}|^2 + G_1 + G_3 + G^S + D_{u_1} + G_{w_1} + D_{w_1} \big)\,ds\\
    &\quad \leq C \|(\phi_0, \psi_0, \omega_0) \|_{L^2}^2 + Ce^{-c\delta \beta}+ C\varepsilon^2 \int_0^t  D_{u_2} \, ds + C(\delta^{3/4}+\varepsilon^2) \int_0^t  D_{w_2} \, ds.
\end{aligned}
\end{equation}
\end{lemma}

\begin{proof}
Multiplying $\eqref{eq:perturbed-eq}_2$ by $\frac{\omega}{\rho}$ and $\eqref{eq:perturbed-eq}_3$ by $\frac{\psi}{\rho}$, and summing the resulting equations, we obtain 
\begin{align*}
(\psi \omega)_t + u (\psi \omega)_x + \frac{1}{\sqrt{\kappa}}\frac{p'(\rho)}{\rho^{1/2}} |\omega|^2
&= \frac{\mu}{\rho} \psi_{xx} \omega + \sqrt{\kappa} \frac{\omega}{\rho} (\rho^{3/2} \omega_x)_x - \sqrt{\kappa} \frac{\psi}{\rho} (\rho^{3/2} \psi_x)_x \\
&\quad + \frac{G \omega + H\psi }{\rho} + \dot{X}(t) (\tilde{u}_x \omega + \tilde{w}_x \psi ).
\end{align*}
We integrate the above equation over $\mathbb{R}_+$ and take integration by parts to get
\begin{equation*} 
\begin{aligned}
    &\frac{d}{dt} \int_{\mathbb{R}_+} \psi \omega \, dx + \frac{1}{\sqrt{\kappa}} \int_{\mathbb{R}_+} \frac{p'(\rho)}{\rho^{1/2}}  |\omega|^2 \,dx + \sqrt{\kappa} \int_{\mathbb{R}_+} \rho^{1/2} |\omega_x|^2 \,dx \nonumber \\
    &\quad =  \left. \left( u \psi \omega - \frac{\mu}{\rho} \psi_x \omega - \sqrt{\kappa} \rho^{1/2} \omega \omega_x + \sqrt{\kappa} \rho^{1/2} \psi \psi_x \right) \right|_{x=0} \nonumber \\
    &\qquad + \int_{\mathbb{R}_+} u_x \psi \omega \,dx - \mu \int_{\mathbb{R}_+}  \psi_x \left( \frac{\omega_x}{\rho} - \frac{\rho_x \omega}{\rho^2} \right) dx \nonumber \\
    &\qquad + \sqrt{\kappa} \int_{\mathbb{R}_+} \rho^{1/2} |\psi_x|^2 \,dx + \sqrt{\kappa} \int_{\mathbb{R}_+} \frac{\rho_x}{\rho^{1/2}} (\omega \omega_x - \psi \psi_x) \,dx \nonumber \\
    &\qquad + \int_{\mathbb{R}_+} \frac{G \omega + H\psi }{\rho} \,dx + \dot{X}(t) \int_{\mathbb{R}_+} (\tilde{u}_x \omega + \tilde{w}_x \psi ) \,dx =:\sum_{i=1}^7 J_i.
\end{aligned}
\end{equation*}

We begin with the boundary term $J_1$. By employing Sobolev interpolation with the smallness assumption \eqref{p-assum}, and Young's inequality, we can bound $J_1$ as
\begin{align*}
|J_1| 
&\le C  \bigl( \|\psi\|_{L^\infty} |\omega(s,0)| + |\psi(s,0)|^{4/3} +  |\omega(s,0)|^{4/3} \bigr) \\
&\quad + C \bigl( \| \psi_x \|_{L^2}^2 \| \psi_{xx}\|_{L^2}^{2} +  \|\omega_x\|_{L^2}^{2} \|\omega_{xx}\|_{L^2}^{2} \bigr)  \\
&\le C  \delta^{4/3} e^{-c\delta|-\sigma s - X(s) - \beta|}  + C \varepsilon^2 \bigl( D_{u_2} + D_{w_2} \bigr).
\end{align*}

For $J_2$, we use $u_x=\psi_x + \tilde{u}_x$ and Young's inequality to obtain
\begin{equation*}
\begin{aligned}
|J_2| \le \|\psi\|_{L^\infty} \int_{\mathbb{R}_+} |\psi_x| |\omega | \, dx + \int_{\mathbb{R}_+} |\tilde{u}_x| |\psi| |\omega| \, dx \le C \varepsilon \bigl( D_{u_1}+ G_{w_1} \bigr) + C \delta ( G^S +G_{w_1} ).
\end{aligned}
\end{equation*}

For $J_3$, \eqref{p-assum} implies $|\rho_x| \le |\tilde{\rho}_x|+|\phi_x| \le C(\delta + \varepsilon)$. Thus,
\begin{align*}
    |J_3| &\le C \int_{\mathbb{R}_+} |\psi_x||\omega_x| \, dx +  C \int_{\mathbb{R}_+} |\rho_x| |\psi_x| |\omega| \, dx\le \frac{1}{16} D_{w_1} + C D_{u_1} + C (\delta + \varepsilon) (D_{u_1} +G_{w_1}).
\end{align*}
 
The term $J_4$ is simply bounded by $ |K_4| \le C D_{u_1} $. For $J_5$, we recall that 
\[\omega = \sqrt{\frac{\kappa}{\rho}}\phi_x +\sqrt{\kappa}\tilde{\rho}_x\left(\frac{1}{\sqrt{\rho}}-\frac{1}{\sqrt{\tilde{\rho}}}\right),\]
which implies $|\phi_x|\le C(|\omega|+|\tilde{\rho}_x||\phi|)$. Then, we again use $|\rho_x|\le |\phi_x|+|\tilde{\rho}_x|\le  C(\delta+\e)$ to derive
\begin{align*}
|J_5| &\le  C \int_{\mathbb{R}_+}(|\phi_x|+|\tilde{\rho}_x| ) |\omega| |\omega_x| \, dx +C \int_{\mathbb{R}_+} (|\phi_x|+|\tilde{\rho}_x| ) |\psi| |\psi_x| \, dx \\
&\le C \int_{\mathbb{R}_+}(\delta+\e) |\omega| |\omega_x| \, dx +C \int_{\mathbb{R}_+} (|\omega|+|\tilde{\rho}_x||\phi|+|\tilde{\rho}_x| ) |\psi| |\psi_x| \, dx \\
&\le C (\delta + \varepsilon)( G_{w_1}+D_{w_1} + G_1+G^S+ D_{u_1}).
\end{align*}
For $J_6$, utilizing the pointwise bound \eqref{eq:GH-est} and \eqref{est-quadratic}, we obtain
\begin{equation*}
\begin{aligned}
|J_6| &\le C \int_{\mathbb{R}_+} |\tilde{u}_x| |(\phi,\psi,\omega)|^2 dx \le C \sqrt{\delta} (G_1+G_3) +CG^S.
\end{aligned}
\end{equation*}
Finally, we estimate $J_7$ by using $|\tilde{w}_x| \le C \delta | \tilde{\rho}_x|$ as
\begin{align*}
    |J_7| \le C \delta^{3/4} |\dot{X}| \bigl( \sqrt{G_3} + \sqrt{G^S} \bigr) \le  \delta |\dot{X}|^2  + C \delta^{1/2} ( G_3 +G^S).
\end{align*}

Combining the estimates for $J_1$ to $J_7$ and integrating over $[0,t]$, we use the smallness of $\delta$ and $\e$ to deduce
\begin{equation} \label{est-highorder1}
\begin{aligned}
& \frac{1}{2}\int_0^t (G_{w_1} + D_{w_1}) \, ds  \\
&\le \frac{1}{2}(\|\psi\|_{L^2}^2+\|\omega\|_{L^2}^2+\|\psi_0\|_{L^2}^2+\|\omega_0\|_{L^2}^2)  + C_1 \int_0^t( \delta |\dot{X}(s)|^2 + D_{u_1} + G^S) \, ds \\
&\quad +C(\sqrt{\delta}+\e)\int_0^t(G_1+G_3)\,ds+C \varepsilon^2 \int_0^t  ( D_{u_2} + D_{w_2})\,ds + C e^{-c \delta \beta},
\end{aligned}
\end{equation}
for some positive constant $C_1$. Multiplying \eqref{est-highorder1} by $\frac{1}{C_1+1}$ and adding it to the $L^2$-estimate \eqref{est-L2}, we obtain the desired result in Lemma~\ref{lem:higeroder-1}.
\end{proof}

To close the estimate, we need to derive the second-order dissipative terms $D_{u_2}$, $G_{w_2}$ and $D_{w_2}$. These terms are derived from the following high-order estimates.

\begin{lemma}\label{lem:higher2}
	Under the assumptions of Proposition \ref{prop:apriori}, for any $\tau \in (0,1)$, there exist positive constants $C, c$, and $C_\tau$ depending on $\tau$ such that, for all $t \in [0,T]$,
	\begin{equation} \label{est-higher-2}
		\begin{aligned}
			&\|(\psi_x, \omega_x)(t,\cdot)\|_{L^2}^2 + \int_0^t D_{u_2} \,ds \\
			&\le  \|(\psi_{0x}, \omega_{0x})\|_{L^2}^2 + \tau \int_0^t D_{w_2} \,ds + C\delta \int_0^t |\dot{X}(s)|^2 \,ds \\
			&\quad + C_\tau \int_0^t \left(G_1 + G_3 + G^S+ D_{u_1} + G_{w_1} + D_{w_1}   \right) ds + C e^{-c\delta \beta}.
		\end{aligned}
	\end{equation}
\end{lemma}

\begin{proof} We multiply $\eqref{eq:perturbed-eq}_2$ by $-\frac{\psi_{xx}}{\rho}$ and $\eqref{eq:perturbed-eq}_3$ by $-\frac{\omega_{xx}}{\rho}$, and then integrate the sum of the resulting two equations over $\mathbb{R}_+$ to get
\begin{equation}\label{eq:higher-order-1}
\begin{aligned}
    &-\int_{\mathbb{R}_+} (\psi_t \psi_{xx} + \omega_t \omega_{xx}) \,dx - \int_{\mathbb{R}_+} u (\psi_x \psi_{xx} + \omega_x \omega_{xx}) \,dx - \frac{1}{\sqrt{\kappa}}\int_{\mathbb{R}_+} \frac{p'(\rho)}{\rho^{1/2}} \omega \psi_{xx} \,dx \\
    &\quad + \mu \int_{\mathbb{R}_+} \frac{|\psi_{xx}|^2}{\rho}  \,dx + \sqrt{\kappa} \int_{\mathbb{R}_+} \frac{\psi_{xx}}{\rho} (\rho^{3/2} \omega_x)_x \,dx - \sqrt{\kappa} \int_{\mathbb{R}_+} \frac{\omega_{xx}}{\rho} (\rho^{3/2} \psi_x)_x \,dx \\
    &=- \int_{\mathbb{R}_+} \frac{G\psi_{xx} + H\omega_{xx}}{\rho} \,dx  - \dot{X}(t) \int_{\mathbb{R}_+} (\tilde{u}_x \psi_{xx} + \tilde{w}_x \omega_{xx}) \,dx.
\end{aligned}
\end{equation}

Applying integration by parts to the time derivative terms and directly expanding the terms including $\sqrt{\kappa}$ to cancel out the highest-order derivatives, \eqref{eq:higher-order-1} can be simplified to
\begin{equation*}
\begin{aligned}
    &\frac{1}{2} \frac{d}{dt} \|(\psi_x, \omega_x)\|_{L^2}^2 + \mu \int_{\mathbb{R}_+} \frac{|\psi_{xx}|^2}{\rho}  \,dx \\
    &= -(\psi_t \psi_x + \omega_t \omega_x)|_{x=0}   + \int_{\mathbb{R}_+} \left[u (\psi_x \psi_{xx} + \omega_x \omega_{xx}) +\frac{1}{\sqrt{\kappa}}\frac{p'(\rho)}{\rho^{1/2}} \omega \psi_{xx} \right] dx \\
    &\quad - \frac{3}{2}\sqrt{\kappa} \int_{\mathbb{R}_+} \frac{\rho_x}{\rho^{1/2}} (\psi_{xx} \omega_x - \omega_{xx} \psi_x) \,dx - \int_{\mathbb{R}_+} \frac{G\psi_{xx} + H\omega_{xx}}{\rho} \,dx \\
    &\quad - \dot{X}(t) \int_{\mathbb{R}_+} (\tilde{u}_x \psi_{xx} + \tilde{w}_x \omega_{xx}) \,dx \\
    &=: K_1 + K_2 + K_3 + K_4 + K_5.
\end{aligned}
\end{equation*}

For the boundary term $K_1$, the boundary conditions $u(t,0) = u_- < 0$ and $w(t,0) = 0$ and \eqref{small-X} imply
\begin{equation} \label{eq:psi_t_omega_t}
    \begin{aligned}
        |\psi_t(t,0)| &= |\tilde{u}_t(t,0)| \le C | \tilde{u}'(-\sigma t -X(t) -\beta)||\sigma + \dot{X}(t)|\le C \delta^2 e^{-c\delta ( \sigma t +\beta)},\\
        |\omega_t(t,0)| &= |\tilde{w}_t(t,0)| \le C | \tilde{w}'(-\sigma t -X(t) -\beta)||\sigma + \dot{X}(t)|\le C \delta^3 e^{-c\delta ( \sigma t +\beta)},\\
        |\omega(t,0)| &= |\tilde{w}(-\sigma t-X(t)-\beta)| \le C|\tilde{\rho}'(-\sigma t -X(t) -\beta)|\le C \delta^2 e^{-c\delta ( \sigma t +\beta)}.
    \end{aligned}
\end{equation}
Utilizing the decay estimates \eqref{eq:psi_t_omega_t} along with the Sobolev inequality, we obtain
\begin{equation*}
\begin{aligned}
|K_1| &\le \left( \| \psi_x\|_{L^\infty} |\psi_t(t,0)|  + \| \omega_x\|_{L^\infty} | \omega_t(t,0)|\right) \\
&\le C \|\psi_x\|_{L^\infty}^2|\psi_t(t,0)|+|\psi_t(t,0)| +\|\omega_x\|_{L^\infty}^2|\omega_t(t,0)|+ |\omega_t(t,0)|\\
&\le C  \delta^2\left( \|\psi_x\|_{L^2}^2+ \|\omega_x\|_{L^2}^2 +\|\psi_{xx}\|_{L^2}^2+\|\omega_{xx}\|_{L^2}^2 \right) + C\delta^2 e^{-c\delta(\sigma t+ \beta)} \\
&\le C \delta^2 (D_{u_1}+D_{w_1}+D_{u_2}+D_{w_2})+ C\delta^2 e^{-c\delta(\sigma t+ \beta)}.
\end{aligned}
\end{equation*}


For $K_2$, applying Young's inequality with a small parameter $\tau \in (0,1)$, we obtain
\begin{equation*}
\begin{aligned}
|K_2| &\le C \int_{\mathbb{R}_+} |\psi_x | | \psi_{xx} | \, dx + C \int_{\mathbb{R}_+} | \omega_x | | \omega_{xx} | \, dx + C \int_{\mathbb{R}_+} | \omega | | \psi_{xx} | \, dx \\
&\le \frac{1}{16} D_{u_2} + \frac{\tau}{4} D_{w_2} + C_\tau (G_{w_1} + D_{u_1} + D_{w_1}),
\end{aligned}
\end{equation*}
where the constant $C_\tau>0$ depends on $\tau$ but is independent of $T,\beta, \delta$ and $\varepsilon$.

For $K_3$, we recall \eqref{p-assum} implies $\|\rho_x\|_{L^\infty} \le C(\delta + \varepsilon)$, which yields
\begin{equation*}
\begin{aligned}
|K_3| &\le C \|\rho_x\|_{L^\infty} \int_{\mathbb{R}_+} \left( |\psi_{xx}||\omega_x| + |\omega_{xx}||\psi_x| \right) dx\le C (\delta + \varepsilon) \left( D_{u_1}+ D_{u_2} + D_{w_1} + D_{w_2} \right).
\end{aligned}
\end{equation*}

We use \eqref{eq:GH-est}, \eqref{est-quadratic} and Young's inequality to estimate $K_4$ as
\begin{equation*}
\begin{aligned}
|K_4| &\le \int_{\mathbb{R}_+}  |\tilde{u}_x| |(\phi,\psi,\omega)| |(\psi_{xx},\omega_{xx})| \, dx \\
&\le C\|\tilde{u}_x\|_{L^\infty}^{1/2} ( D_{u_2} +  D_{w_2} )+ C \|\tilde{u}_x\|_{L^\infty}^{1/2} \int_{\mathbb{R}_+} |\tilde{u}_x| |(\phi,\psi,\omega)|^2 \, dx \\
&\le C \delta \bigl(D_{u_2}+D_{w_2} +G_1+G_3 + G^S \bigr).
\end{aligned}
\end{equation*}

For $K_5$, we similarly obtain
\begin{equation*}
\begin{aligned}
|K_5| \le \delta |\dot{X}(t)|^2 + C \delta (D_{u_2} + D_{w_2} ).
\end{aligned}
\end{equation*}


Combining the estimates from $K_1$ to $K_5$ and integrating over $[0,t]$, we use the smallness of $\delta$ and $\varepsilon$ to deduce
\begin{equation*}
\begin{aligned}
&\|(\psi_x, \omega_x)(t,\cdot)\|_{L^2}^2 + \int_0^t D_{u_2} \,ds \\
&\le  \|(\psi_{0x}, \omega_{0x})\|_{L^2}^2 + \tau \int_0^t D_{w_2} \,ds + C\delta \int_0^t |\dot{X}(s)|^2 \,ds \\
&\quad + C_\tau \int_0^t \left(G_1 + G_3 + G^S+ D_{u_1} + G_{w_1} + D_{w_1}   \right) ds + C \delta e^{-c\delta \beta},
\end{aligned}
\end{equation*}
which is the desired estimate.

\end{proof}

To recover the dissipation $D_{w_2}$ and close the estimate \eqref{est-higher-2}, we perform an additional cross-estimate. 


\begin{lemma}
	Under the assumptions of Proposition \ref{prop:apriori}, there exist positive constants $C, c$, and $C_2$ such that, for all $t \in [0,T]$,
	\begin{equation} \label{est-higher-3}
		\begin{aligned}
		  \frac{1}{2}\int_0^t(G_{w_2}+D_{w_2})\,ds &\le \bigl(\|(\psi_x,\omega_x)(t,\cdot)\|_{L^2}^2+\|(\psi_{0x},\omega_{0x})\|_{L^2}^2 \bigr)+C_2\int_0^t(D_{u_1}+D_{u_2})\,ds\\
			&\quad+C\delta\int_0^t|\dot{X}(s)|^2\,ds +C(\delta+\e)\int_0^t(G_1+G_3+G^S+D_{u_1})\,ds +C e^{-c\delta \beta}.
		\end{aligned}
	\end{equation}
\end{lemma}

\begin{proof}

Multiplying $\eqref{eq:perturbed-eq}_2$ by $-\frac{\omega_{xx}}{\rho}$ and $\eqref{eq:perturbed-eq}_3$ by $-\frac{\psi_{xx}}{\rho}$, and integrating the sum of the resulting equations over $\mathbb{R}_+$, we have
\begin{equation}\label{eq:cross-first-1}
\begin{aligned}
    &-\int_{\mathbb{R}_+} (\psi_t \omega_{xx} + \omega_t \psi_{xx}) \,dx - \int_{\mathbb{R}_+} u (\psi_x \omega_{xx} + \omega_x \psi_{xx}) \,dx - \frac{1}{\sqrt{\kappa}}\int_{\mathbb{R}_+} \frac{p'(\rho)}{\rho^{1/2}} \omega \omega_{xx} \,dx \\
    &= - \mu \int_{\mathbb{R}_+} \frac{\psi_{xx}}{\rho}  \omega_{xx} \,dx - \sqrt{\kappa} \int_{\mathbb{R}_+} \frac{\omega_{xx}}{\rho} (\rho^{3/2} \omega_x)_x \,dx + \sqrt{\kappa} \int_{\mathbb{R}_+} \frac{\psi_{xx}}{\rho} (\rho^{3/2} \psi_x)_x \,dx \\
    &\quad  -\int_{\mathbb{R}_+} \frac{G \omega_{xx}+H \psi_{xx}}{\rho}  \,dx - \dot{X}(t)\int_{\mathbb{R}_+} \bigl(  \tilde{u}_x \omega_{xx} + \tilde{w}_x \psi_{xx}\bigr) \,dx.
\end{aligned}
\end{equation}

Applying integration by parts, 
\eqref{eq:cross-first-1} can be rearranged as
\begin{equation*}
\begin{aligned}
    &\frac{d}{dt} \int_{\mathbb{R}_+} \psi_x \omega_x \,dx + \frac{1}{\sqrt{\kappa}}\int_{\mathbb{R}_+} \frac{p'(\rho)}{\rho^{1/2}} |\omega_{x}|^2 \,dx +\sqrt{\kappa} \int_{\mathbb{R}_+}  \rho^{1/2} |\omega_{xx}|^2 \,dx \\
    &= -\Bigl(\psi_t \omega_x + \omega_t \psi_x + \frac{1}{\sqrt{\kappa}} \frac{p'(\rho)}{\rho^{1/2}} \omega \omega_x \Bigr) \Big|_{x=0} + \int_{\mathbb{R}_+} u (\psi_x \omega_{xx} + \omega_x \psi_{xx}) \,dx \\
    &\quad -\mu \int_{\mathbb{R}_+} \frac{\psi_{xx}}{\rho}  \omega_{xx} \,dx -\frac{1}{\sqrt{\kappa}} \int_{\mathbb{R}_+} \left( \frac{p'(\rho)}{\rho^{1/2}} \right)_x \omega \omega_x \,dx  \\
    &\quad - \frac{3}{2}\sqrt{\kappa} \int_{\mathbb{R}_+} \frac{\rho_x}{\rho^{1/2}} (\omega_x \omega_{xx} - \psi_x \psi_{xx}) \,dx + \sqrt{\kappa} \int_{\mathbb{R}_+} \rho^{1/2} |\psi_{xx}|^2 \,dx \\
    &\quad -\int_{\mathbb{R}_+} \frac{G \omega_{xx}+H \psi_{xx}}{\rho}  \,dx - \dot{X}(t)\int_{\mathbb{R}_+} \bigl(  \tilde{u}_x \omega_{xx} + \tilde{w}_x \psi_{xx}\bigr) \,dx =: \sum_{i=1}^8 L_i.
\end{aligned}
\end{equation*}

We estimate the right-hand side terms $L_1, \dots, L_8$ as follows. 

For $L_1$, utilizing the decay estimates \eqref{eq:psi_t_omega_t} along with the Sobolev inequality, we obtain
\begin{equation*}
\begin{aligned}
|L_1| &\le   \left( \| \omega_x\|_{L^\infty} |\psi_t(t,0)|  + \| \psi_x\|_{L^\infty} | \omega_t(t,0)|+ |\omega(t,0)| \|\omega_x\|_{L^\infty}\right)  \\
&\le C  \delta^2 e^{-c \delta (\sigma t +\beta)}\left( \|\psi_x\|_{L^2}^{1/2}  \|\psi_{xx}\|_{L^2}^{1/2} + \|\omega_x\|_{L^2}^{1/2} \|\omega_{xx}\|_{L^2}^{1/2} \right) \\
&\le C \varepsilon^2 \left( D_{u_2}  + D_{w_2} \right)  + C \delta^{8/3} e^{-c\delta(\sigma t+\beta)}.
\end{aligned}
\end{equation*}

For $L_2$ and $L_3$, we apply Young's inequality to obtain
\begin{equation*}
\begin{aligned}
|L_2| \le \frac{1}{16} (G_{w_2} + D_{w_2})  + C ( D_{u_1} + D_{u_2}),\quad\mbox{and}\quad |L_3| \le  \frac{1}{16}  D_{w_2}  + C  D_{u_2}.
\end{aligned}
\end{equation*}

For $L_4$ and $L_5$, the estimate $\|\rho_x\|_{L^\infty} \le C(\delta + \varepsilon)$ implies
\begin{equation*}
\begin{aligned}
|L_4| &\le C \|\rho_x\|_{L^\infty} \int_{\mathbb{R}_+} |\omega| |\omega_x| \, dx \le  C(\delta + \varepsilon) \left( G_{w_1} + G_{w_2} \right),
\end{aligned}
\end{equation*}
and 
\begin{equation*}
\begin{aligned}
|L_5| &\le C \|\rho_x\|_{L^\infty} \int_{\mathbb{R}_+} (|\omega_x| |\omega_{xx}| + |\psi_x| |\psi_{xx}| ) \, dx \le  C(\delta + \varepsilon) \left( G_{w_2} + D_{w_2} + D_{u_1} + D_{u_2}\right).
\end{aligned}
\end{equation*}

The term $L_6$ is simply bounded by $|L_6| \le C D_{u_2}$. For $L_7$, we use \eqref{eq:GH-est}, \eqref{est-quadratic} and Young's inequality to obtain
\begin{equation*}
\begin{aligned}
|L_7| &\le \int_{\mathbb{R}_+}  |\tilde{u}_x| |(\phi,\psi,\omega)| |(\psi_{xx},\omega_{xx})| \, dx \\
&\le C\|\tilde{u}_x\|_{L^\infty}^{1/2} ( D_{u_2} +  D_{w_2} )+ C \|\tilde{u}_x\|_{L^\infty}^{1/2} \int_{\mathbb{R}_+} |\tilde{u}_x| |(\phi,\psi,\omega)|^2 \, dx \\
&\le  C \delta \bigl( G_1+G_3 + G^S + D_{u_2} + D_{w_2} \bigr).
\end{aligned}
\end{equation*}

Finally, we estimate $L_8$ as before:
\begin{equation*}
\begin{aligned}
|L_8| \le \delta |\dot{X}(t)|^2 + C \delta^{1/2} (D_{u_2} + D_{w_2} ).
\end{aligned}
\end{equation*}

Combining the estimates from $L_1$ to $L_8$ and integrating over $[0,t]$, we deduce
\begin{equation*}
\begin{aligned}
&\int_{\mathbb{R}_+} \psi_x \omega_x \,dx + \frac{1}{2}\int_0^t (G_{w_2}+D_{w_2}) \,ds \\
&\le  \int_{\R_+}\psi_{0x}\omega_{0x}\,dx  + C \int_0^t (D_{u_1}+D_{u_2}) \, ds + \delta \int_0^t |\dot{X}(s)|^2 \,ds \\
&\quad + C(\delta+\varepsilon) \int_0^t \left( G_1 + G_3 + G^S  + D_{u_1} \right) ds + C \delta e^{-c\delta \beta}.
\end{aligned}
\end{equation*}


Thus, there exist positive constants $C, C_2 > 0$ such that
\begin{align*}
	\frac{1}{2}\int_0^t(G_{w_2}+D_{w_2})\,ds &\le \bigl(\|(\psi_x,\omega_x)(t,\cdot)\|_{L^2}^2+\|(\psi_{0x},\omega_{0x})\|_{L^2}^2 \bigr) +C_2\int_0^t(D_{u_1}+D_{u_2})\,ds\\
	&\quad+C\delta\int_0^t|\dot{X}(s)|^2\,ds +C(\delta+\e)\int_0^t(G_1+G_3+G^S+D_{u_1})\,ds +C\delta e^{-c\delta \beta},
\end{align*}
which is the desired estimate.

\end{proof}





\noindent \textbf{Proof of Lemma \ref{lem:H1}.} 
Multiplying \eqref{est-higher-3} by $\frac{1}{C_2+1}$ and adding it to \eqref{est-higher-2}, we obtain
\begin{align*}
	&\|(\psi_x,\omega_x)(t,\cdot)\|_{L^2}^2+\int_0^t D_{u_2}\,ds + \frac{1}{2(C_2+1)}\int_0^t (G_{w_2}+D_{w_2})\,ds\\
	&\le C\|(\psi_{0x},\omega_{0x})\|_{L^2}^2 + \frac{1}{C_2+1}\|(\psi_x,\omega_x)\|_{L^2}^2 +\tau\int_0^t D_{w_2}\,ds \\
	&\quad + \frac{C_2}{C_2+1}\int_0^t (D_{u_1}+D_{u_2})\,ds + C\delta\int_0^t|\dot{X}(s)|^2\,ds \\
	&\quad +C_\tau\int_0^t (G_1+G_3+G^S+D_{u_1}+G_{w_1}+D_{w_1})\,ds +C\delta e^{-c\delta\beta}.
\end{align*}
Choosing $\tau = \frac{1}{4(C_2+1)}$, there exists a positive constant $C_3$ such that
\begin{equation} \label{est-higher-combine-1}
\begin{aligned}
&  \|(\psi_x, \omega_x)(t,\cdot)\|_{L^2}^2 + \int_0^t   (D_{u_2}+G_{w_2}+D_{w_2}) \,ds\\
&\le  C \|(\psi_{0x}, \omega_{0x})\|_{L^2}^2 + C \delta e^{-c\delta \beta}  \\
&\quad + C_3 \int_0^t \left(\delta|\dot{X}(s)|^2 + G_1 + G_3 + G^S+ D_{u_1} + G_{w_1} + D_{w_1}   \right) ds.
\end{aligned}
\end{equation}
Multiplying \eqref{est-higher-combine-1} by $\frac{1}{C_3+1}$ and adding it to \eqref{est-L2-combined}, and using the smallness of $\delta$ and $\varepsilon$, we finally derive \eqref{est:H1}. This completes the proof of Lemma~\ref{lem:H1}.

\begin{appendix}
\section{Proof of time-asymptotic behavior}\label{sec:app-A}

In this appendix, we prove that $g$ defined in \eqref{def:g} satisfies $g\in W^{1,1}(\R_+)$, completing the proof of time-asymptotic behavior \eqref{asym-U}. First of all, we note that
\[|(\rho-\tilde{\rho})_x|\le C|w-\tilde{w}|+C|\tilde{\rho}_x||\rho-\tilde{\rho}|.\]
Therefore, we get
\[\int_0^\infty |g(t)|\,dt \le \int_0^\infty (G_{w_1}+G_1+G^S+D_{u_1})\,dt<+\infty.\]
Next, we will show that $g'(t)\in L^1(\R_+)$. We split $\|g'\|_{L^1}$ as
\begin{align*}
	\int_0^\infty |g'(t)|\,dt&\le \int_0^\infty \left|\int_{\R_+}\phi_x\phi_{xt}\,dx\right|+\left|\int_{\R_+}\psi_x\psi_{xt}\,dx\right|+\left|\int_{\R_+}\phi_{xx}\phi_{xxt}\,dx\right|\,dt\\
	&=:I_{1}+I_{2}+I_{3}.
\end{align*}

We first estimate $I_1$. It follows from \eqref{eq:perturbed-eq} that
\[\phi_{xt} = -u\phi_{xx}-\rho\psi_{xx}+F_1+\dot{X}\tilde{\rho}_{xx},\]
where $F_1$ is defined by
\[
F_1 := F_x + u_x \phi_x + \rho_x \psi_x,
\]
and satisfies
\[|F_1|\le C\big(\delta |\tilde{u}_x||(\phi,\psi)|+|\tilde{u}_x||(\phi_x,\psi_x)|+|\phi_x\psi_x|\big).\]
Therefore, we split $I_1$ as
\begin{align*}
	I_1&\le\int_0^\infty \left|\int_{\R_+}u\phi_x\phi_{xx}\,dx\right|+\left|\int_{\R_+}\rho\psi_{xx}\phi_x\,dx\right|+\left|\int_{\R_+}F_1\phi_x\,dx\right|+|\dot{X}|\left|\int_{\R_+}\phi_x\tilde{\rho}_{xx}\,dx\right|\,dt\\
	&=:I_{11}+I_{12}+I_{13}+I_{14}.
\end{align*}
It is straightforward to observe that
\[I_{11}+I_{12}\le \int_0^\infty \|\phi_x\|_{H^1}^2+\|\psi_{xx}\|_{L^2}^2\,dt\le\int_0^\infty (G_1+G^S+G_{w_1}+D_{w_1}+D_{u_2})\,dt.\]
Furthermore, we estimate $I_{13}$ using \eqref{p-assum} as
\begin{align*}
	I_{13}&\le C\int_0^\infty \int_{\R_+}(\delta|\tilde{u}_x||(\phi,\psi)|+|\tilde{u}_x||(\phi_x,\psi_x)+|\phi_x\psi_x|)|\phi_x|\,dx\,dt\\
	&\le C\int_0^\infty (G_1+G^S+G_{w_1}+D_{u_1})\,dt.
\end{align*}
Finally, $I_{14}$ is bounded as
\[I_{14}\le C\delta\int_0^\infty |\dot{X}|^2\,ds +C\int_0^\infty (G_1+G^S+G_{w_1})\,dt.\]
Thus, combining the estimate of $I_{11},\ldots,I_{14}$, we have $|I_1|<+\infty$. Next, we control $I_2$. After integration by parts, we have
\begin{align*}
	|I_2|&\le \int_0^\infty |\psi_x(t,0)||\psi_t(t,0)|\,dt+\int_0^\infty\left|\int_{\R_+}\psi_{xx}\psi_{t}\,dx\right|\,dt\\
	&\le C+C\int_0^\infty (D_{u_1}+D_{u_2})\,dt+C\int_0^\infty\int_{\R_+}|\psi_t|^2\,dx\,dt.
\end{align*}
Using the equation \eqref{eq:perturbed-eq}, we get
\begin{align*}
	\int_0^\infty \int_{\R_+}|\psi_t|^2\,dx\,dt&\le C\int_0^\infty (D_{u_1}+G_{w_1}+D_{u_2}+D_{w_1}+D_{w_2}+G_1+G^S)\,dt+C\delta\int_0^\infty |\dot{X}|^2\,dt\\
	&< +\infty.
\end{align*}
This implies $|I_{2}|<+\infty$. Finally, we again use integration by parts to modify $I_3$ as
\begin{align}
\begin{aligned}\label{I3}	
	I_3\le\left|\int_{\R_+}\phi_{xx}\phi_{xxt}\,dx\right| &\le \left|\int_{\R_+}\phi_{xxx}\phi_{xt}\,dx\right|+|\phi_{xx}(t,0)\phi_{xt}(t,0)|\\
	&\le\left|\int_{\R_+}\phi_{xxx}\phi_{xt}\,dx\right|+\|\phi_{xx}\|_{L^\infty}|\tilde{\rho}_x(t,0)|.
\end{aligned}
\end{align}
Since the smallness assumption \eqref{p-assum} implies
\begin{align*}
	|\phi_{xx}| &= \left|\left(\sqrt{\frac{\rho}{\kappa}}\omega-\sqrt{\rho}\tilde{\rho}_x\left(\frac{1}{\sqrt{\rho}}-\frac{1}{\sqrt{\tilde{\rho}}}\right)\right)_x\right|\le C(|\rho_x\omega|+|\omega_x|+|\tilde{\rho}_x||\rho_x||\phi|+|\tilde{\rho}_x|(|\phi_x|+|\phi|))\\
	&\le C(|\phi_x||\omega|+|\tilde{\rho}_x||\omega|+|\omega_x|+|\tilde{\rho}_x||\phi|+|\tilde{\rho}_x||\phi_x||\phi|)\\
	&\le C(|\omega|+|\omega_x|+|\tilde{\rho}_x||\phi|),
\end{align*}
and
\begin{align*}
	|\phi_{xxx}|&\le C(|\rho_{xx}\omega|+|\rho_x||\omega_x|+|\omega_{xx}|+|\rho_{xx}||\tilde{\rho}_x||\phi|+|\tilde{\rho}_{xxx}||\phi|+|\tilde{\rho}_x|(|\phi_{xx}|+|\phi_x|+|\phi|)\\&
	\quad + |\rho_x||\tilde{\rho}_{xx}||\phi|+|\rho_x||\tilde{\rho}_x|(|\phi_x|+|\phi|)+|\tilde{\rho}_{xx}|(|\phi_x|+|\phi|)\\
	&\le C(|\omega|+|\phi_{xx}|+|\omega_x|+|\phi_x||\omega_x|+|\omega_{xx}|+|\phi_{xx}|+|\tilde{\rho}_x||\phi|+|\phi_x|+|\phi_x|^2)\\
	&\le C(|\omega|+|\omega_x|+|\omega_{xx}|+|\tilde{\rho}_x||\phi|)<+\infty,
\end{align*}
we can bound the second term of the right-hand side of \eqref{I3} as

\begin{align*}
	\int_0^\infty \|\phi_{xx}(t)\|_{L^\infty}|\tilde{\rho}_x(t,0)|\,dt&\le C\int_0^\infty \|\phi_{xx}\|_{H^1}^2\,dt +\int_0^\infty |\tilde{\rho}_x(t,0)|^2\,dt\\
	&\le\int_0^\infty (G_{w_1}+D_{w_1}+G_1+G^S)\,dt+C\delta e^{-c_0\delta\beta}.
\end{align*}
Finally, using a similar argument as in the estimate of $I_1$, we obtain
\begin{align*}
	&\int_0^\infty \left|\phi_{xxx}\phi_{xt}\,dx \right|\,dt\\
    &\quad \leq \int_0^\infty\left|\int_{\R_+}u\phi_{xx}\phi_{xxx}\,dx\right| +\left|\int_{\R_+}\rho\psi_{xx}\phi_{xxx}\,dx\right|+\left|\int_{\R_+}F_1\phi_{xxx}\,dx\right|+|\dot{X}|\left|\int_{\R_+}\phi_{xxx}\tilde{\rho}_{xx}\,dx\right|\,dt\\
	&\quad \le\int_0^\infty (G_1+G^S+G_{w_1}+D_{u_1}+D_{u_2}+D_{w_1}+D_{w_2})\,dt+C\delta\int_0^\infty |\dot{X}|^2\,dt<+\infty.
\end{align*}
Thus, we conclude that $g\in W^{1,1}(\R_+)$, and this implies the desired time-asymptotic behavior \eqref{asym-U}.

\end{appendix}

\bibliography{reference}

\begin{thebibliography}{10}

\bibitem{BDD06} S. Benzoni-Gavage, R. Danchin, and S. Descombes, Well-posedness of one-dimensional Korteweg models, Electron. J. Differ. Equations 2006 (2006), Paper No. 5.

\bibitem{CL21} Z.-Z. Chen and Y.-P. Li, Z. Chen and Y. Li, Asymptotic behavior of solutions to an impermeable wall problem of the compressible fluid models of Korteweg type with density-dependent viscosity and capillarity, SIAM J. Math. Anal. {\bf 53} (2021), no.~2, 1434--1473; MR4228314

\bibitem{CLS19} Z. Chen, Y. Li and M. Sheng, Asymptotic stability of viscous shock profiles for the 1D compressible Navier-Stokes-Korteweg system with boundary effect, Dyn. Partial Differ. Equ. {\bf 16} (2019), no.~3, 225--251; MR4008883

\bibitem{D96} C. M. Dafermos, Entropy and the stability of classical solutions of hyperbolic systems of conservation laws, in {\it Recent mathematical methods in nonlinear wave propagation}, 48--69, Lecture Notes in Math., 1640, Springer-Verlag, Berlin, 1996. 

\bibitem{D79} R. J. DiPerna, Uniqueness of solutions to hyperbolic conservation laws, {\it Indiana Univ. Math. J.}, 28 (1979), no. 1, 137--188. 

\bibitem{DS85} J.~E. Dunn and J.~B. Serrin Jr., On the thermomechanics of interstitial working, Arch. Rational Mech. Anal. {\bf 88} (1985), no.~2, 95--133; MR0775366

\bibitem{HK26} S. Han and J. Kim, Time-asymptotic stability of composite wave for the one-dimensional compressible fluid of Korteweg type, SIAM J. Math. Anal. {\bf 58} (2026), no.~1, 547--597; MR5025407

\bibitem{HKKL25-JDE} S. Han et al., Long-time behavior towards viscous-dispersive shock for Navier-Stokes equations of Korteweg type, J. Differential Equations {\bf 426} (2025), 317--387; MR4854625

\bibitem{HMS03} F. Huang, A. Matsumura and X. Shi, Viscous shock wave and boundary layer solution to an inflow problem for compressible viscous gas, Comm. Math. Phys. {\bf 239} (2003), no.~1-2, 261--285; MR1997442

\bibitem{Hong20} H. Hong, Stationary solutions to outflow problem for 1-D compressible fluid models of Korteweg type: existence, stability and convergence rate, Nonlinear Anal. Real World Appl. {\bf 53} (2020), 103055, 27 pp.; MR4033130

\bibitem{Hong22} H. Hong, Stability of stationary solutions and viscous shock wave in the inflow problem for isentropic Navier-Stokes-Korteweg system, J. Differential Equations {\bf 314} (2022), 518--573; MR4369180

\bibitem{HKKL25-JMAA} X. Huang et al., Asymptotic behavior toward viscous shock for impermeable wall and inflow problems of barotropic Navier-Stokes equations, J. Math. Anal. Appl. {\bf 552} (2025), no.~2, Paper No. 129803, 50 pp.; MR4925912

\bibitem{HLO26-AA} X. Huang, H. Lee, and H. Oh, Stability of viscous shock for the Navier--Stokes--Fourier system: outflow and impermeable wall problems, {\it Anal. Appl.}, (2026), 1--52. 

\bibitem{KOW25} M.-J. Kang, H. Oh, and Y. Wang, Asymptotic behavior toward viscous shocks for the outflow problem of barotropic Navier--Stokes equations, arXiv preprint arXiv:2505.08171 (2025). 

\bibitem{KV17} M.-J. Kang and A.~F. Vasseur, $L^2$-contraction for shock waves of scalar viscous conservation laws, Ann. Inst. H. Poincar\'e{} C Anal. Non Lin\'eaire {\bf 34} (2017), no.~1, 139--156; MR3592682

\bibitem{KV21} M.-J. Kang and A.~F. Vasseur, Contraction property for large perturbations of shocks of the barotropic Navier-Stokes system, J. Eur. Math. Soc. (JEMS) {\bf 23} (2021), no.~2, 585--638; MR4195742

\bibitem{KV-Inven} M.-J. Kang and A.~F. Vasseur, Uniqueness and stability of entropy shocks to the isentropic Euler system in a class of inviscid limits from a large family of Navier-Stokes systems, Invent. Math. {\bf 224} (2021), no.~1, 55--146; MR4228501

\bibitem{KVW23} M.-J. Kang, A.~F. Vasseur and Y. Wang, Time-asymptotic stability of composite waves of viscous shock and rarefaction for barotropic Navier-Stokes equations, Adv. Math. {\bf 419} (2023), Paper No. 108963, 66 pp.; MR4560999

\bibitem{KNZ03} S. Kawashima, S. Nishibata and P. Zhu, Asymptotic stability of the stationary solution to the compressible Navier-Stokes equations in the half space, Comm. Math. Phys. {\bf 240} (2003), no.~3, 483--500; MR2005853

\bibitem{KZ08} S. Kawashima and P. Zhu, Asymptotic stability of nonlinear wave for the compressible Navier-Stokes equations in the half space, J. Differential Equations {\bf 244} (2008), no.~12, 3151--3179; MR2420517

\bibitem{KZ09} S. Kawashima and P. Zhu, Asymptotic stability of rarefaction wave for the Navier-Stokes equations for a compressible fluid in the half space, Arch. Ration. Mech. Anal. {\bf 194} (2009), no.~1, 105--132; MR2533925

\bibitem{Kor01} D. J. Korteweg, Sur la forme que prennent les \'equations du mouvement des fluides si l'on tient compte des forces capillaires par des variations de densit\'e, Arch. Neerl. Sci. Exactes Nat. Ser. II 6 (1901), 1--24.

\bibitem{Kotschote08} M. Kotschote, Strong solutions for a compressible fluid model of Korteweg type, Ann. Inst. H. Poincar\'e{} C Anal. Non Lin\'eaire {\bf 25} (2008), no.~4, 679--696; MR2436788

\bibitem{Kotschote10} M. Kotschote, Strong well-posedness for a Korteweg-type model for the dynamics of a compressible non-isothermal fluid, J. Math. Fluid Mech. {\bf 12} (2010), no.~4, 473--484; MR2749439

\bibitem{LC22} Y. Li and Z. Chen, Large-time behavior of solutions to an inflow problem for the compressible Navier-Stokes-Korteweg equations in the half space, J. Math. Fluid Mech. {\bf 24} (2022), no.~4, Paper No. 103, 24 pp.; MR4487786



\bibitem{LTY22} Y. Li, J. Tang and S. Yu, Asymptotic stability of rarefaction wave for the compressible Navier-Stokes-Korteweg equations in the half space, Proc. Roy. Soc. Edinburgh Sect. A {\bf 152} (2022), no.~3, 756--779; MR4430951


\bibitem{LXC23} Y. Li, R. Xu and Z. Chen, Asymptotic stability of a nonlinear wave for the compressible Navier-Stokes-Korteweg equations in the half space, Z. Angew. Math. Phys. {\bf 74} (2023), no.~4, Paper No. 167, 27 pp.; MR4620251

\bibitem{LZ21} Y. Li and P. Zhu, Asymptotic stability of the stationary solution to an out-flow problem for the Navier-Stokes-Korteweg equations of compressible fluids, Nonlinear Anal. Real World Appl. {\bf 57} (2021), Paper No. 103193, 23 pp.; MR4130092

\bibitem{MatBVP} A. Matsumura, Inflow and outflow problems in the half space for a one-dimensional isentropic model system of compressible viscous gas, Nonlinear Anal. {\bf 47} (2001), no.~6, 4269--4282; MR1972365

\bibitem{MM99} A. Matsumura and M. Mei, Convergence to travelling fronts of solutions of the $p$-system with viscosity in the presence of a boundary, Arch. Ration. Mech. Anal. {\bf 146} (1999), no.~1, 1--22; MR1682659

\bibitem{MN01} A. Matsumura and K. Nishihara, Large-time behaviors of solutions to an inflow problem in the half space for a one-dimensional system of compressible viscous gas, Comm. Math. Phys. {\bf 222} (2001), no.~3, 449--474; MR1888084

\bibitem{NZ09} T.~T. Nguyen and K.~R. Zumbrun, Long-time stability of large-amplitude noncharacteristic boundary layers for hyperbolic-parabolic systems, J. Math. Pures Appl. (9) {\bf 92} (2009), no.~6, 547--598; MR2565843

\bibitem{Seppecher93} P. Seppecher, Equilibrium of a Cahn-Hilliard fluid on a wall: influence of the wetting properties of the fluid upon the stability of a thin liquid film, European J. Mech. B Fluids {\bf 12} (1993), no.~1, 69--84; MR1204995

\bibitem{Seppecher96} P. Seppecher, Moving contact lines in the Cahn-Hilliard theory, Int. J. Eng. Sci. 34 (1996), 977--992. \url{https://doi.org/10.1016/0020-7225(95)00141-7}

\bibitem{VdW94} J. D. van der Waals, Thermodynamische Theorie der Kapillarit\"at unter Voraussetzung stetiger Dichte\"anderung, {\it Z. Phys. Chem.}, 13 (1894), no. 1, 657--725. 

\end{thebibliography}

\end{document}